%% file: The_Shape_of_a_Random_Numerical_Semigroup.tex
\documentclass[12pt]{amsart}

\usepackage{amsthm}
\usepackage[usenames,dvipsnames]{xcolor}
\usepackage{graphicx}
\usepackage{stmaryrd} 
\usepackage{tikz}
\usepackage{tikz-cd}
\usepackage{ bbold }
\usetikzlibrary{arrows}
\usepackage{verbatim}
\usepackage{amssymb,amsfonts}
\usepackage[all,arc]{xy}
\usepackage{enumerate}
\usepackage{mathrsfs, mathabx}
\usepackage{xcolor}[2007/01/21]
\usepackage{hyperref}
\usepackage{soul}
\usepackage{esint} 
\usepackage{multirow} 
\usepackage[margin=1.0in]{geometry}
\usepackage[
   open,
   openlevel=2,
   atend
 ]{bookmark}[2011/12/02]
\usepackage{graphics}
\usepackage{enumitem,kantlipsum}
\usepackage{graphicx}

\usepackage{float}

\usepackage{subcaption}

\makeatletter
\renewcommand\p@subfigure{\thefigure}        
\makeatother

\def\S{{\mathcal S}}

\def\B{{\mathcal B}}

\def\E{{\mathbb E}}

\renewcommand{\phi}{\varphi}

\renewcommand{\P}{\mathbb{P}}
\newcommand{\cC}{\mathcal{C}}

\makeatletter
\def\@pmods#1{\mkern4mu({\operator@font mod}\mkern 6mu#1)}
\makeatother

\def\N{{\mathbb N}}
\def\R{{\mathbb R}}

\def\Ima{{\operatorname{Im}}}

\def\SS{{\mathcal S}}

\def\Var{{\operatorname{Var}}}

\newcommand{\Frob}{\text{Frob}}

\newtheorem{question}{Question}
\newtheorem{theorem}{Theorem}

\newtheorem{lemma}[theorem]{Lemma}

\newtheorem{proposition}[theorem]{Proposition}
\newtheorem{corollary}[theorem]{Corollary}

\begin{document}

\title{The shape of a random numerical semigroup}

\author{Maria Bras-Amor\'{o}s}
\address{Universitat Politècnica de Catalunya}
\email{maria.bras@upc.edu}

\author{Nathan Kaplan}
\address{Department of Mathematics, University of California, Irvine}
\email{nckaplan@math.uci.edu}

\author{Deepesh Singhal}
\address{Department of Mathematics, University of California, Irvine}
\email{singhald@uci.edu}



\begin{abstract}
We study statistical properties of random numerical semigroups of a given genus.  We analyze the graph of a typical numerical semigroup, understood as a function from $\N$ to $\N$.  If $S$ is a numerical semigroup of genus $g$, this leads us to consider the collection of points $\left(\frac{k-1}{g-1},\frac{a_k(S)}{g} \right)$ where $1 \le k \le g$ and $a_k(S)$ denotes the $k$\textsuperscript{th} smallest nonzero element of $S$.  We show that as $g \rightarrow \infty$, this set of points typically becomes closer to a union of two line segments.
We prove analogous results for numerical semigroups ordered by Frobenius number.
\end{abstract}
\maketitle

\section{Introduction}

Let $\N_0 = \{0,1,2,\ldots\}$ denote the nonnegative integers. A \emph{numerical semigroup} $S$ is a subset $S \subseteq \N_0$ that contains $0$, is closed under addition, and has finite complement in $\N_0$.  The size of the complement, $|\N_0 \setminus S|$, is the \emph{genus} of $S$ and is denoted by $g(S)$.  The largest element of $\N_0 \setminus S$ is the \emph{Frobenius number} of $S$ and is denoted by $F(S)$.  Since $F(S) \not\in S$, for any positive integer $x$, at most one element of the pair $\{x,F(S)-x\}$ is contained in $S$.  This implies $F(S) \le 2g(S)-1$.  For a general reference on numerical semigroups, see \cite{sanchez2009numerical}.

Let $\SS_g$ denote the set of numerical semigroups of genus $g$.  A numerical semigroup $S \in \SS_g$ is uniquely determined by a set of $g$ gaps in the interval $[1,2g-1]$, or equivalently, by a set of $g$ elements of $S$ contained in $[1,2g]$.  Let $a_k(S)$ be the $k$\textsuperscript{th} smallest nonzero element of $S$.  The element $a_1(S)$ is the \emph{multiplicity} of $S$ and is denoted by $m(S)$.  Note that for any positive integer $g$ and for any $S \in \SS_g$, we have $a_g(S) = 2g$.

\begin{question}
As we vary over all semigroups $S \in \SS_g$, how are the points $(k, a_k(S))$ typically distributed?
\end{question}

We give an example with $g=55$ to demonstrate what we have in mind.  There are $1,142,140,736,859$ numerical semigroups in $\SS_{55}$.  Each gives a collection of $55$ integer points $(k,a_k(S))$ contained in the rectangle defined by $1 \le x \le 55$ and $0\le y \le 110$.  

Instead of working directly with the sets of points $(k,a_k(S))$ in a $g \times 2g$ rectangle, we study a scaled version where the points given by $\left(\frac{k-1}{g-1}, \frac{a_k(S)}{g} \right)$ are in the rectangle where $0 \le x \le 1,\ 0 \le y \le 2$.  This allows us to consider different values of $g$ on the same scale.

Figure \eqref{fig:genus55} gives a heat map where the color of the $\frac{1}{55}\times \frac{1}{55}$ square with lower-left corner at $(\frac{k-1}{54},\frac{j}{55})$ reflects the number of $S \in \SS_{55}$ with $a_k(S) = j$.  The main motivation for this paper is to understand these kinds of pictures as $g \rightarrow \infty$.

\begin{figure}[ht]\label{Fig1}
  \centering
  \begin{subfigure}{0.49\textwidth}
    \centering
    \includegraphics[width=\linewidth]{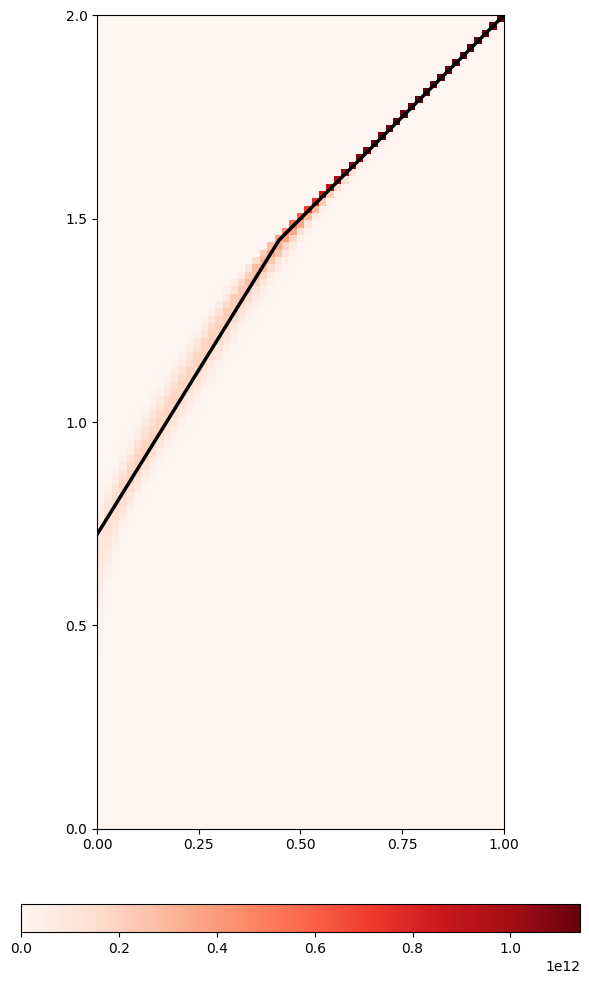}
    \caption{Genus 55}
    \label{fig:genus55}
  \end{subfigure}\hfill
  \begin{subfigure}{0.49\textwidth}
    \centering
    \includegraphics[width=\linewidth]{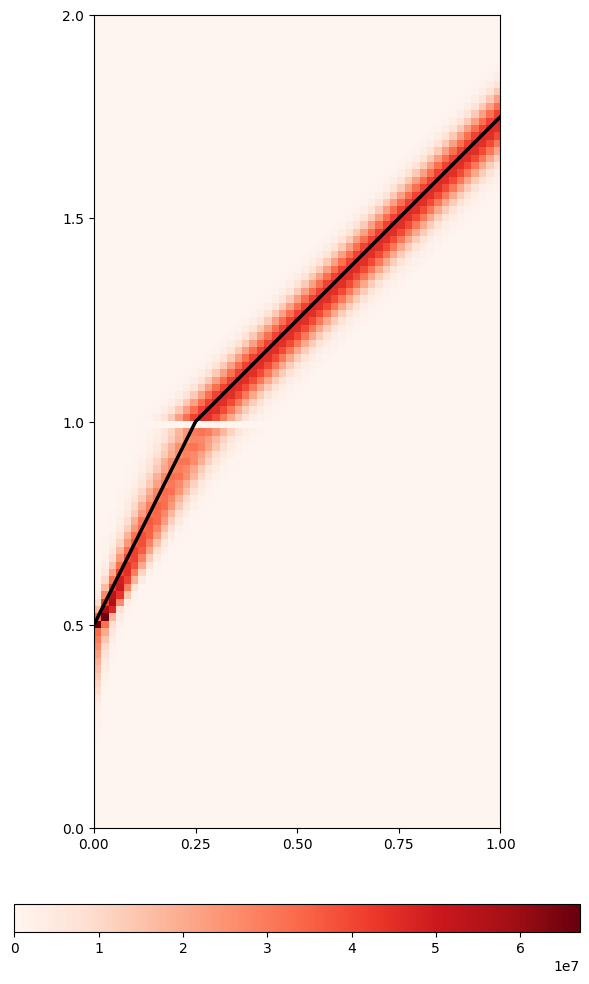}
    \caption{Frobenius number 55}
    \label{fig:frob55}
  \end{subfigure}
  \caption{Heat maps of $\psi_S$ and $\Xi_S$.}
  \label{fig:genus-frob}
\end{figure}

\subsection{Main Result}

Let $S$ be a numerical semigroup.  For a real number $0 \leq \alpha \leq 1$, let $t = 1 + \lfloor \alpha(g(S)-1) \rfloor$ and define
\[
\psi_S(\alpha) = \frac{a_t(S)}{g(S)}.
\]
We see that 
\[
\psi_S(0) = \frac{a_1(S)}{g(S)} = \frac{m(S)}{g(S)}\ \ \ \text{and }\ \ \ \psi_S(1) = \frac{a_{g(S)}}{g(S)} = 2.
\]
It is clear that $\psi_S(\alpha)$ is a nondecreasing function from $[0,1]$ to $[0,2]$.

Throughout this paper, we write $\varphi = \frac{1+\sqrt{5}}{2}$ for the golden ratio.  We often refer to the constant $\gamma = \frac{5+\sqrt{5}}{10} = \frac{\varphi}{\sqrt{5}}$.  Define the function 
\[
\psi(\alpha) = \begin{cases}
    \gamma +\phi \alpha &\text{if } 0\leq \alpha\leq \frac{1}{\sqrt{5}}\\
    1 + \alpha & \text{if } \frac{1}{\sqrt{5}} \leq \alpha \leq 1
\end{cases}.
\] 
Note that $\psi$ is a monotonically increasing continuous function from $[0,1]$ to $[0,2]$.

Throughout the paper, when we describe a numerical semigroup by listing its elements, we use the symbol $\rightarrow$ to indicate that it contains all larger elements.  Figure~\ref{fig:graphpsi} shows the graph of $\psi$ and the graph of $\psi_S$ for three different examples of semigroups of genus $33$: 
\begin{enumerate}[align=left, leftmargin=*]
    \item the ordinary semigroup $\{0,34\rightarrow\}$ (\ref{fig:psi_ordinary}) 
    ;
    \item the hyperelliptic semigroup $\{0,2,4,\ldots,66\rightarrow\}$ (\ref{fig:psi_hyperelliptic});
    \item the well-tempered semigroup $\{0,12,19,24,28,31,34,36,38,40,42,43,45\rightarrow\}$ (\ref{fig:psi_welltempered}).
\end{enumerate}

Notice that the ordinary semigroup and the hyperelliptic semigroup of each genus are extremal in the sense that the graph of any other semigroup of the same genus is contained between these two graphs.

\begin{figure}[ht]
  \centering
  \begin{subfigure}{0.3\textwidth}
    \centering
    \includegraphics[width=\linewidth]{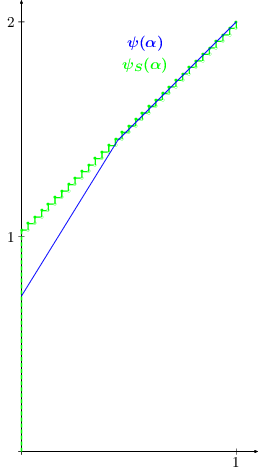}
    \caption{Ordinary semigroup}
    \label{fig:psi_ordinary}
  \end{subfigure}\hfill
  \begin{subfigure}{0.3\textwidth}
    \centering
    \includegraphics[width=\linewidth]{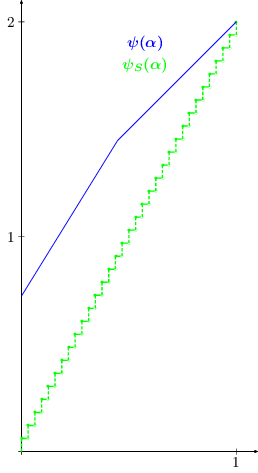}
    \caption{Hyperelliptic semigroup}
    \label{fig:psi_hyperelliptic}
  \end{subfigure}\hfill
  \begin{subfigure}{0.3\textwidth}
    \centering
    \includegraphics[width=\linewidth]{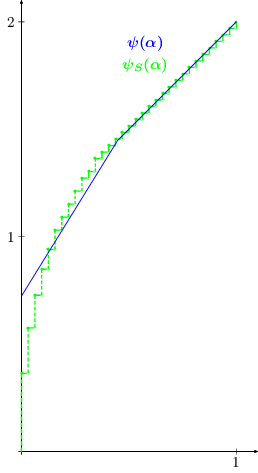}
    \caption{Well tempered semigroup}
    \label{fig:psi_welltempered}
  \end{subfigure}
  \caption{Graph of $\psi$ and $\psi_S$}
 \label{fig:graphpsi}
\end{figure}

We denote the uniform probability distribution on the finite set $\SS_g$ by $\P_g$.  Throughout the paper we write $n_g = |\SS_g|$.  Our main result is as follows.
\begin{theorem}\label{Thm: Uniform conv psi_S}
As $g\to\infty$, $\psi_S\to \psi$ uniformly in probability.
\end{theorem}
This means that for every $\epsilon>0$,
\[
\lim_{g\to\infty} \P_g\Big[ \sup_{0\leq \alpha\leq 1}|\psi_S(\alpha) -\psi(\alpha)| >\epsilon \Big] =0.
\]

Our second main result is an analogue of Theorem \ref{Thm: Uniform conv psi_S} for numerical semigroups ordered by Frobenius number.  Let $\S_{\Frob}(F)$ denote the set of numerical semigroups with Frobenius number $F$ and let $N_{\Frob}(F) = |\S_{\Frob}(F)|$.  Denote the uniform probability distribution on $\S_{\Frob}(F)$ by $\P_{\Frob,F}$.

Let $S$ be a numerical semigroup.  For a real number $0 \leq \alpha \leq 1$, let $t = 1 + \lfloor \alpha (F(S)-1) \rfloor$ and define
\[
\Xi_S(\alpha) = \frac{a_t(S)}{F(S)}.
\]
We see that 
\[
\Xi_S(0) = \frac{a_1(S)}{F(S)} = \frac{m(S)}{F(S)}\ \ \ \text{and }\ \ \ \Xi_S(1) = \frac{a_{F(S)}(S)}{F(S)} \leq 2.
\]
It is clear that $\Xi_S(\alpha)$ is a nondecreasing function from $[0,1]$ to $[0,2]$.

Define the function 
\[
\Xi(\alpha) = \begin{cases}
    \frac{1}{2} + 2\alpha &\text{if } 0\leq \alpha\leq \frac{1}{4}\\
    \frac{3}{4} + \alpha & \text{if } \frac{1}{4} \leq \alpha \leq 1
\end{cases}.
\]
Note that $\Xi$ is a continuous and monotonically increasing function from $[0,1]$ to $[0,2]$.

Figure~\ref{fig:graphXi} shows the graph of $\Xi$ and the graph of $\Xi_S$ for three different examples of semigroups with Frobenius number $45$: 
\begin{enumerate}[align=left, leftmargin=*]
\item the ordinary semigroup $\{0, 46\rightarrow\}$ (\ref{fig:Xi_ordinary});
\item the hyperelliptic semigroup $\{0,2,\ldots,44,46\rightarrow\}$ (\ref{fig:Xi_hyperelliptic}); 
\item the almost-well-tempered semigroup $\{0,12,19,24,28,31,34,36,38,40,42,43,44,46\rightarrow~\}$ ~(\ref{fig:Xi_almostwelltempered}).
\end{enumerate}

\begin{figure}[H]
  \centering
  \begin{subfigure}{0.3\textwidth}
    \centering
    \includegraphics[width=\linewidth]{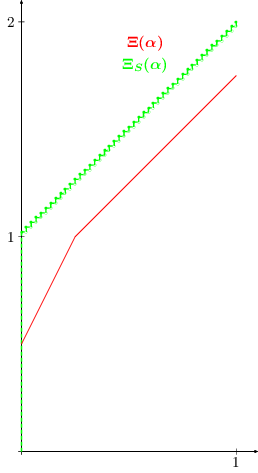}
    \caption{Ordinary semigroup\\ \ }
    \label{fig:Xi_ordinary}
  \end{subfigure}\hfill
  \begin{subfigure}{0.3\textwidth}
    \centering
    \includegraphics[width=\linewidth]{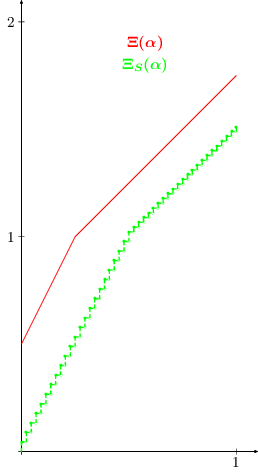}
    \caption{Hyperelliptic semigroup\\ \ }
    \label{fig:Xi_hyperelliptic}
  \end{subfigure}\hfill
  \begin{subfigure}{0.3\textwidth}
    \centering
    \includegraphics[width=\linewidth]{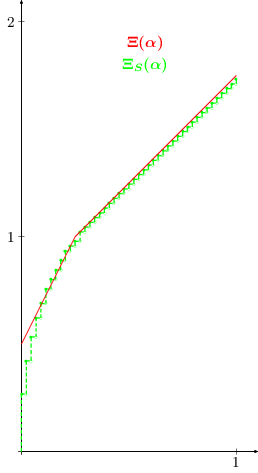}
    \caption{Almost well-tempered semigroup}
    \label{fig:Xi_almostwelltempered}
  \end{subfigure}
  \caption{Graph of $\Xi$ and $\Xi_S$}
 \label{fig:graphXi}
\end{figure}

As in the case of $\psi_S$, the ordinary semigroup and the hyperelliptic semigroup of each Frobenius number are extremal in the sense that the graph of $\Xi_S$ of any other semigroup of the same Frobenius number is contained between these two graphs.

Figure \eqref{fig:frob55} gives a heat map where the color of the $\frac{1}{55}\times \frac{1}{55}$ square with lower-left corner at $(\frac{k-1}{54},\frac{j}{55})$ reflects the number of $S \in \SS_{\Frob}(55)$ with $a_k(S) = j$.
We will show the following.
\begin{theorem}\label{Thm: Uniform Convergence by frob}
As $F\to\infty$, $\Xi_S\to \Xi$ uniformly in probability.
\end{theorem}

These results fit into a literature concerning limit shapes of random combinatorial objects.  To give one notable example, there has been an extensive study of shapes of Young diagrams of large random integer partitions under different distributions.  For more information on these kinds of results, see, for example, the notes of Okounkov \cite{Okounkov}.  

\subsection{Outline of the paper}

Our main result also contributes to the study of statistical properties of the finite collection of numerical semigroups of genus $g$. This has been a topic of active research interest in recent years.  In Section \ref{sec:review_Sg} we review the results in this area that we will need and describe a framework for counting numerical semigroups in $\SS_g$ with certain properties. We also do the same for counting numerical semigroups in $\SS_{\Frob}(F)$.

In Section \ref{sec:uniform_pointwise} we discuss some basic tools from probability theory and explain how Theorem~\ref{Thm: Uniform conv psi_S} follows from a statement about the pointwise convergence in probability as $g \rightarrow \infty$ of the functions $\psi_S(\alpha)$ to the function $\psi(\alpha)$.  In Section \ref{Sec: pointwise} we prove the pointwise convergence result that we need.
In Section~\ref{Sec: proof for frob}, we follow a similar strategy for numerical semigroups ordered by Frobenius number and prove Theorem~\ref{Thm: Uniform Convergence by frob}.

\section{Computational results}

Before going into the proofs of Theorem \ref{Thm: Uniform conv psi_S} and Theorem \ref{Thm: Uniform Convergence by frob}, we discuss some computations supporting these results.  Figure \eqref{fig:genus55} makes a compelling visual case that the points $\left(\frac{k-1}{g-1},\frac{a_k(S)}{g}\right)$ typically are `close' to the function $\psi(\alpha)$.  We can make more precise statements using the language of linear regression; see, for example, \cite{montgomery2021introduction}.

For each semigroup $S$ of genus $g$ we consider the set of points $(\frac{k-1}{g-1},\frac{a_k(S)}{g})$
for $k=1,\ldots,g$. Observe that these points are on the graph of $\psi_S(\alpha)$.
Let $i_{cut}(g)=\lceil\frac{g-1}{\sqrt{5}}\rceil$.
We want to compute separately
\begin{itemize}
\item  the regression line of the points $(\frac{k-1}{g-1},\frac{a_k(S)}{g})$ with $1\leq k\leq i_{cut}(g)$,
\item  the regression line of the points $(\frac{k-1}{g-1},\frac{a_k(S)}{g})$ with $i_{cut}(g)+1\leq k\leq g$.
\end{itemize}
We will call the first line the \emph{left line}, and the second line the \emph{right line}.
We say that we \emph{gain a point on the left} if $i_{cut}(g)=i_{cut}(g-1)+1$ and we say that we \emph{gain a point on the right} if $i_{cut}(g)=i_{cut}(g-1)$.
  
For the computation of the left line, we define
\begin{align*}
\bar x^l(S)& =\frac{1}{2}\left(\frac{i_{cut}(g)-1}{g-1}\right), \\
E^l(S)& =\frac{1}{i_{cut}(g)}\sum_{i=1}^{i_{cut}(g)}\frac{a_i(S)}{g}.
\end{align*}
Then the slope and the $y$-intercept of the left line are, respectively, given by
\begin{eqnarray*}
m^l(S)&=&\frac{\sum_{i=1}^{i_{cut}(g)}\left(E^l(S)-\frac{a_i(S)}{g}\right)\left(\bar x^l(S)-\frac{i-1}{g-1}\right)}{\sum_{i=1}^{i_{cut}(g)}\left(\bar x^l(S)-\frac{i-1}{g-1}\right)^2},\\  
  b^l(S)&=&E^l(S)-m^l(S)\bar x^l(S).
  \end{eqnarray*}
The accuracy of this approximation is measured by its \emph{coefficient of determination}
  $$(R^2)^l(S)=1-\frac{\sum_{i=1}^{i_{cut}(g)}\left(m^l(S)\frac{i-1}{g-1}+b^l(S)-\frac{a_i(S)}{g}\right)^2}{\sum_{i=1}^{i_{cut}(g)}\left(E^l(S)-\frac{a_i(S)}{g}\right)^2}.$$
  
We proceed similarly for the right line.  We define 
\begin{align*}
\bar x^r(S)& =\frac{1}{2}\left(\frac{g-1+i_{cut}(g)}{g-1}\right),\\
E^r(S) &=\frac{1}{g-i_{cut}(g)}\sum_{i=i_{cut}(g)+1}^{g}\frac{a_i(S)}{g},
\end{align*}
and
\begin{eqnarray*}
  m^r(S)&=&\frac{\sum_{i=i_{cut}(g)+1}^{g}\left(E^r(S)-\frac{a_i(S)}{g}\right)\left(\bar x^r(S)-\frac{i-1}{g-1}\right)}{\sum_{i=i_{cut}(g)+1}^{g}\left(\bar x^r(S)-\frac{i-1}{g-1}\right)^2},\\
  b^r(S)&=&E^r(S)-m^r(S)\bar x^r(S),\\
  (R^2)^r(S)&=&1-\frac{\sum_{i=i_{cut}(g)+1}^{g}\left(m^r(S)\frac{i-1}{g-1}+b^r(S)-\frac{a_i(S)}{g}\right)^2}{\sum_{i=i_{cut}(g)+1}^{g}\left(E^r(S)-\frac{a_i(S)}{g}\right)^2}.
\end{eqnarray*}

In order to consider these values taken over the entire set of all the semigroups of genus $g$, we compute the mean values
\[
\begin{array}{rclrcl}
 \bar m^l(g)&=&\frac{1}{n_g}\sum_{S\in \SS_g} m^l(S),&
\bar m^r(g)&=&\frac{1}{n_g}\sum_{S\in \SS_g}m^r(S),\\
 \bar b^l(g)&=&\frac{1}{n_g}\sum_{S\in \SS_g} b^l(S),&
\bar b^r(g)&=&\frac{1}{n_g}\sum_{S\in \SS_g}b^r(S),\\
(\bar R^2)^l(g)&=&\frac{1}{n_g}\sum_{S\in \SS_g} (R^2)^l(S),&
(\bar R^2)^r(g)&=&\frac{1}{n_g}\sum_{S\in \SS_g}(R^2)^r(S).\\
\end{array}
\]
We include data for these values in Appendix \ref{sec:appendix}.  In Table~\ref{tab:regressionparameters}, for each genus from $4$ to $52$ we give the number of semigroups of that genus, the number of points used for the computation of the left line and the number of points used for the computation of the right line of each semigroup $S$, and the mean values $\bar m^l(g)$, $\bar b^l(g)$, $(\bar R^2)^l(g)$, $\bar m^r(g)$, $\bar b^r(g)$, and $(\bar R^2)^r(g)$. In Table~\ref{tab:regressionincrements}, we give the increments of these values for each genus with respect to the previous one.

\begin{figure}[htbp]
  \centering
  \includegraphics[width=0.75\textwidth]{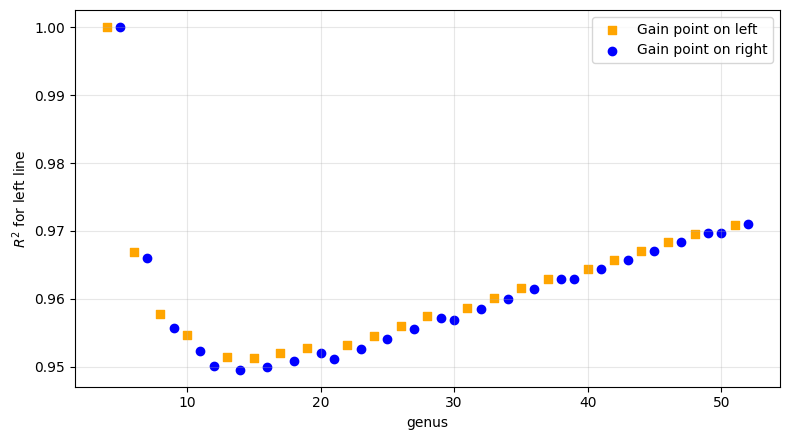}
  \caption{$\bar R^2$ values for the left line.}
  \label{fig:R2-left genus}
\end{figure}

Notice that for the left line the $\bar R^2$ value is generally increasing beyond $g=14$. However, we see that $(\bar R^2)^l(g)>(\bar R^2)^l(g-1)$ when $i_{cut}(g)=i_{cut}(g-1)+1$, while $(\bar R^2)^l(g)<(\bar R^2)^l(g-1)$ when $i_{cut}(g)=i_{cut}(g-1)$.
\begin{figure}[htbp]
  \centering
  \includegraphics[width=0.75\textwidth]{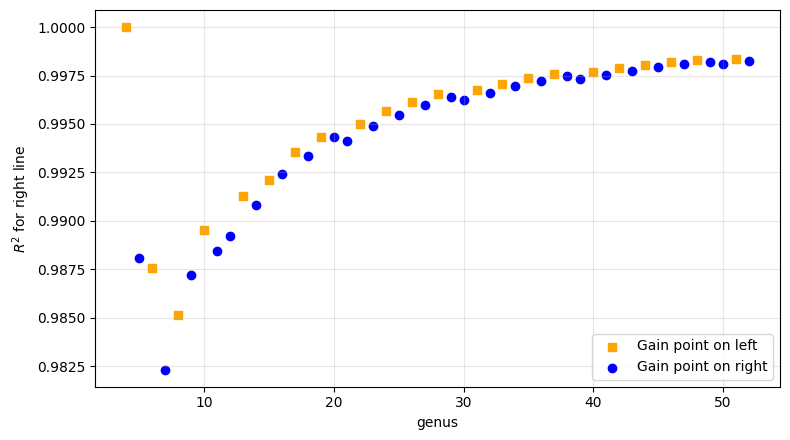}
  \caption{$\bar R^2$ values for the right line.}
  \label{fig:R2-right genus}
\end{figure}
For the right line, we see a general increase in $\bar R^2$ values beyond $g=7$. We also notice that $(\bar R^2)^r(g)<(\bar R^2)^r(g-1)$ when $i_{cut}(g)=i_{cut}(g-1)+1$, while $(\bar R^2)^r(g)>(\bar R^2)^r(g-1)$ when $i_{cut}(g)=i_{cut}(g-1)$.

\begin{figure}[htbp]
  \centering
  \includegraphics[width=0.65\textwidth]{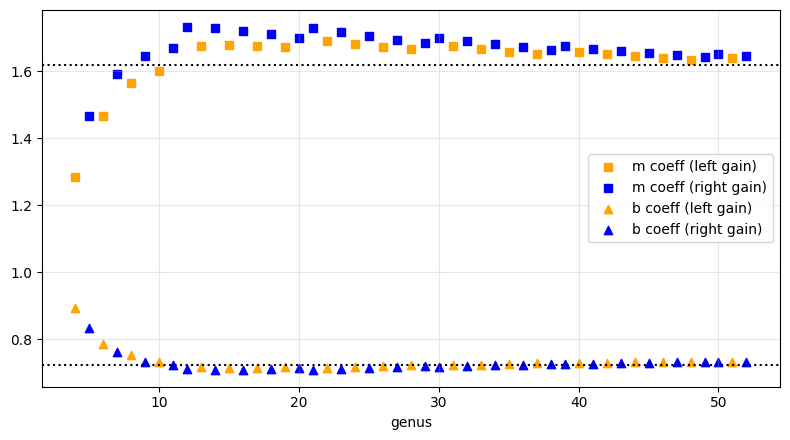}
  \caption{$\bar m$ and $\bar b$ coefficients for the left line.}
  \label{fig:left-coeff}
\end{figure}

\begin{figure}[htbp]
  \centering
  \includegraphics[width=0.65\textwidth]{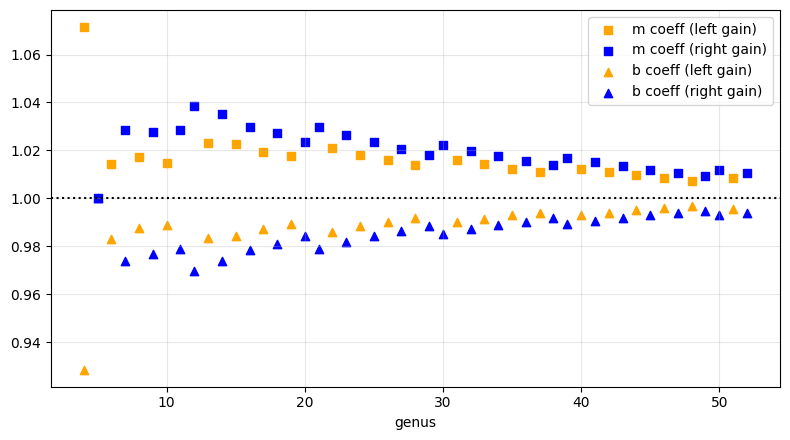}
  \caption{$\bar m$ and $\bar b$ coefficients for the right line.}
  \label{fig:right-coeff}
\end{figure}

We now consider semigroups ordered by Frobenius number.  For each semigroup $S$ with Frobenius number $F$ we consider the set of points $\left(\frac{k-1}{F-1},\frac{a_k(S)}{F}\right)$
for $k=1,\ldots,F$. Observe that these points are on the graph of $\Xi_S(\alpha)$.
Define $i_{cut}(F)=\lfloor\frac{F-1}{4}\rfloor+1$.
Just as we did for semigroups ordered by genus, for each semigroup with Frobenius number $F$ we compute
\begin{itemize}
\item  the regression line of the points $\left(\frac{k-1}{F-1},\frac{a_k(S)}{F}\right)$ with $1\leq k\leq i_{cut}(F)$,
\item  the regression line of the points $\left(\frac{k-1}{F-1},\frac{a_k(S)}{F}\right)$ with $i_{cut}(F)+1\leq k\leq F$.
\end{itemize}
We define
\begin{align*}
\bar x^l(S) &=\frac{1}{2}\left(\frac{i_{cut}(F)-1}{F-1}\right),\\
\bar x^r(S) &= \frac{1}{2} \left(\frac{F-1+i_{cut}(F)}{F-1}\right),\\
E^l(S) &=\frac{1}{i_{cut}(F)}\sum_{i=1}^{i_{cut}(F)}\frac{a_i(S)}{F}, \\
E^r(S) & =\frac{1}{F-i_{cut}(F)}\sum_{i=i_{cut}(F)+1}^{F}\frac{a_i(S)}{F}.
\end{align*}
The slopes and the $y$-intercepts of the left and right lines are given by
\begin{eqnarray*}
  m^l(S)&=&\frac{\sum_{i=1}^{i_{cut}(F)}\left(E^l(S)-\frac{a_i(S)}{F}\right)\left(\bar x^l(S)-\frac{i-1}{F-1}\right)}{\sum_{i=1}^{i_{cut}(F)}\left(\bar x^l(S)-\frac{i-1}{F-1}\right)^2},\\  
    b^l(S)&=&E^l(S)-m^l(S)\bar x^l(S),\\
  m^r(S)&=&\frac{\sum_{i=i_{cut}(F)+1}^{F}\left(E^r(S)-\frac{a_i(S)}{F}\right)\left(\bar x^r(S)-\frac{i-1}{F-1}\right)}{\sum_{i=i_{cut}(F)+1}^{F}\left(\bar x^r(S)-\frac{i-1}{F-1}\right)^2},\\
  b^r(S)&=&E^r(S)-m^r(S)\bar x^r(S).
  \end{eqnarray*}
We compute the coefficients of determination for each of the two lines,
\begin{eqnarray*}
(R^2)^l(S)&=&1-\frac{\sum_{i=1}^{i_{cut}(F)}\left(m^l(S)\frac{i-1}{F-1}+b^l(S)-\frac{a_i(S)}{F}\right)^2}{\sum_{i=1}^{i_{cut}(F)}\left(E^l(S)-\frac{a_i(S)}{F}\right)^2},\\
  (R^2)^r(S)&=&1-\frac{\sum_{i=i_{cut}(F)+1}^{F}\left(m^r(S)\frac{i-1}{F-1}+b^r(S)-\frac{a_i(S)}{F}\right)^2}{\sum_{i=i_{cut}(F)+1}^{F}\left(E^r(S)-\frac{a_i(S)}{F}\right)^2}.
\end{eqnarray*}
We also compute the mean values
$$\begin{array}{rclrcl}
  \bar m^l(F)&=&\frac{1}{|\SS_{Frob}(F)|}\sum_{S\in \SS_{Frob}(F)} m^l(S),&
\bar m^r(F)&=&\frac{1}{|\SS_{Frob}(F)|}\sum_{S\in \SS_{Frob}(F)}m^r(S),\\
 \bar b^l(F)&=&\frac{1}{|\SS_{Frob}(F)|}\sum_{S\in \SS_{Frob}(F)} b^l(S),&
\bar b^r(F)&=&\frac{1}{|\SS_{Frob}(F)|}\sum_{S\in \SS_{Frob}(F)}b^r(S),\\
 (\bar R^2)^l(F)&=&\frac{1}{|\SS_{Frob}(F)|}\sum_{S\in \SS_{Frob}(F)} (R^2)^l(S),&
(\bar R^2)^r(F)&=&\frac{1}{|\SS_{Frob}(F)|}\sum_{S\in \SS_{Frob}(F)}(R^2)^r(S).\\
\end{array}$$
In Table~\ref{tab:regressionparameters_xi} in Appendix \ref{sec:appendix}, for each Frobenius number from $6$ to $52$, we give the number of semigroups with that Frobenius number, the number of points used for the computation of the left line and the number of points used for the computation of the right line of each semigroup $S$, and the mean values $\bar m^l(F)$, $\bar b^l(F)$, $(\bar R^2)^l(F)$, $\bar m^r(F)$, $\bar b^r(F)$, and $(\bar R^2)^r(F)$. In Table~\ref{tab:regressionincrements_xi} we give the increments of these values for each Frobenius number with respect to the previous one.

\begin{figure}[htbp]
  \centering
  \includegraphics[width=0.75\textwidth]{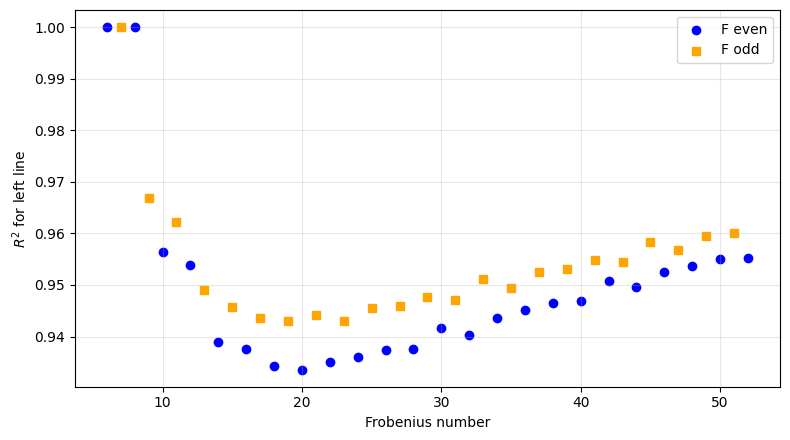}
  \caption{$\bar R^2$ values for the left line, by Frobenius number.}
  \label{fig:R2-left Frob}
\end{figure}

We note that for the left line, the $\bar R^2$ value is generally increasing beyond $F=20$. We see that $\bar R^2$ increases whenever $F$ is even and that $\bar R^2$ decreases whenever $F$ is odd.

\begin{figure}[htbp]
  \centering
  \includegraphics[width=0.75\textwidth]{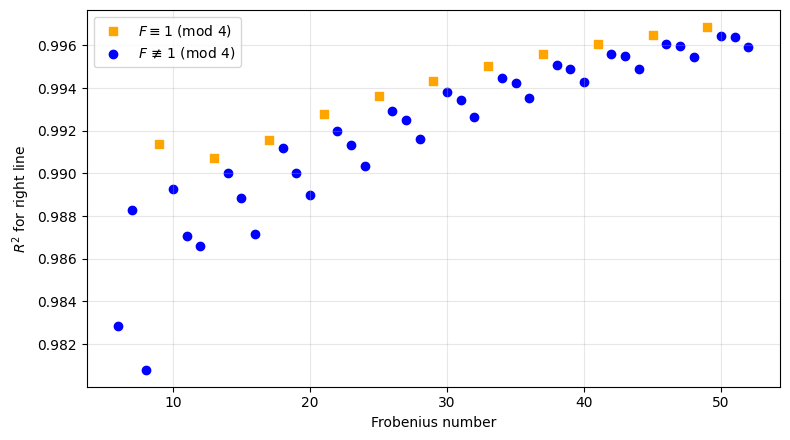}
  \caption{$\bar R^2$ values for the right line, by Frobenius number.}
  \label{fig:R2-right Frob}
\end{figure}

For the right line, the $\bar R^2$ value is generally increasing beyond $F=8$.
We see that the value of $\bar R^2$ increases when $F\equiv 1 \pmod{4}$ and decreases otherwise.

Producing all of this data is time-consuming, as is exploring the tree of all numerical semigroups up to a given genus. There has been extensive work on efficiently exploring the semigroup tree, see for example \cite{FH,ba2024,ba_unleaved}. Here we used the seeds algorithm given in \cite{bf2018,ba2024} and its more recent refinement for the unleaved tree \cite{ba_unleaved}.
In order to generate these graphs, we not only count these semigroups, but we must collect data from each semigroup to measure its contribution to the heat maps and to the regression lines. The overall complexity is the reason why it is difficult to extend the previous figures for larger values of $g$ or for larger values of $F$.

\section{Background}\label{sec:review_Sg}

\subsection{Properties of a typical numerical semigroup of genus $g$}

There is a significant body of literature studying the finite collection of numerical semigroups of genus $g$.  Kaplan gives an introduction to this topic in the survey \cite{KaplanCounting}.  It is natural to ask about the size of this finite set.  Recall that $n_g = |\SS_g|$.  In \cite{BA_Fibonacci}, Bras-Amor\'os computed the first $50$ terms of this sequence and made a conjecture about the asymptotic growth of $n_g$.  This conjecture was proven by Zhai.
\begin{theorem}\cite[Theorem 1]{zhai2013fibonacci}\label{thm_zhai}
There exists a constant $S> 3.78$ such that
\[
\lim_{g\to\infty}\frac{n_g}{\phi^g} = S.
\]
\end{theorem}

A major step in Zhai's proof of this result is to prove a conjecture of Zhao \cite[Conjecture~4.1]{zhao2010constructing}, which states informally that most numerical semigroups of genus $g$ satisfy $F(S) < 3m(S)$.  Kaplan and Ye use Theorem \ref{thm_zhai} to sharpen this result. 
\begin{theorem}\cite[Theorem 4]{kaplan2013proportion}\label{thm_KY_F2m}
For fixed $\epsilon > 0$, we have
\[
\lim_{g\to\infty}\P_{g}[|F(S)-2m(S)|< \epsilon g]=1.
\]
\end{theorem}

Singhal has strengthened this result in \cite{singhal2022distribution}.
\begin{theorem}\cite[Theorem 8]{singhal2022distribution}\label{F 2m}
Given $\epsilon>0$, there is an $M(\epsilon) >0$ such that for all $g > 0$ we have
\[
\P_g[|F(S)-2m(S)|>M(\epsilon)]<\epsilon.
\]
\end{theorem}

Zhu has shown that numerical semigroups with $F(S)$ significantly larger than $2m(S)$ occur with exponentially vanishing probability.  
\begin{theorem}\cite[Theorem 1.7]{zhu2023sub}
Fix some $0 < \epsilon < 1$. Then,
\[
\limsup_{g\to\infty} 
\big|\{S\in \S_g: F(S)>(2+\epsilon)m(S)\}\big|^\frac{1}{g} <\phi.
\]
\end{theorem}

We have a detailed understanding of the typical size of $F(S)$ for numerical semigroups in $\SS_g$. It is natural to ask for analogous results for other invariants.  
Kaplan and Ye study the typical size of $m(S)$ for a numerical semigroup in $\SS_g$.
\begin{theorem}\cite[Proposition 16]{kaplan2013proportion}\label{Kaplan Mult}
For fixed $\epsilon > 0$, we have
\[
\lim_{g\to\infty}\P_{g}[|m(S)-\gamma g |< \epsilon g]=1.
\]
\end{theorem}
Theorem \ref{Kaplan Mult} and Theorem \ref{thm_KY_F2m} together imply that for fixed $\epsilon>0$,
\begin{equation}\label{eqn:FG}
    \lim_{g\to\infty}\P_{g}[|F(S)-2\gamma g |< \epsilon g]=1.
\end{equation}

In \cite{kaplan2023expected}, Kaplan and Singhal study the typical size of several other invariants of numerical semigroups in $\SS_g$.  They also study the probability that a particular positive integer $n$ is contained in a randomly chosen semigroup in $\SS_g$.  Define a step function $f_1:[0,2]\setminus\{\gamma,2\gamma\}\to [0,1]$ by
\[
f_1(x) = 
\begin{cases}
    0  & \mbox{if } 0 \leq x <\gamma \\
    \frac{\sqrt{5}-1}{2}  & \mbox{if } \gamma< x<2\gamma \\
    1 & \mbox{if } 2\gamma<x\leq 2.
\end{cases}
\]
They show that $f_1\left(\frac{n}{g}\right)$ is a good approximation for the probability that $n$ is contained in a random element of $\SS_g$, and also that these probabilities for various $n$ are asymptotically independent.  This result is closely related to our Theorem \ref{Thm: Uniform conv psi_S}. The nonzero values of $f_1$ are $\phi^{-1}$ and $1$, which are the reciprocals of the slopes of the function $\psi$.
\begin{theorem}\cite[Theorem 26]{kaplan2023expected}
Fix $\epsilon_1, \epsilon_2 > 0$.
\begin{enumerate}
\item There exists an $M(\epsilon_1, \epsilon_2)>0$ such that for all $g>M(\epsilon_1, \epsilon_2)$ and
\[
n \in \big[1,(\gamma-\epsilon_1)g\big) \cup \big((\gamma+\epsilon_1)g,(2\gamma-\epsilon_1)g\big) \cup \big((2\gamma+\epsilon_1)g,2g\big]
\]
we have
\[
\left|\P_g[n \in S]-f_{1}\left(\frac{n}{g}\right)\right|<\epsilon_2.
\]

\item There exists an $M(\epsilon_1, \epsilon_2)>0$ such that for all $g>M(\epsilon_1, \epsilon_2)$ and
\[
i,j \in \big[1,(\gamma-\epsilon_1)g\big) \cup \big((\gamma+\epsilon_1)g,(2\gamma-\epsilon_1)g\big) \cup \big((2\gamma+\epsilon_1)g,2g\big]
\]
with $i\neq j$, we have
\[
\left|\P_g[\{i,j\}\subseteq S]-f_{1}\left(\frac{i}{g}\right)f_{1}\left(\frac{j}{g}\right)\right|<\epsilon_2.
\]
\end{enumerate}
\end{theorem}

\subsection{A decomposition of the set of numerical semigroups of genus $g$}

In this subsection we explain a useful way to divide up the set of numerical semigroups of genus $g$ that we will use in the proof of Theorem \ref{Thm: Uniform conv psi_S}.  The \emph{depth} of a numerical semigroup $S$ is $q(S)=\left\lceil\frac{F(S)}{m(S)}\right\rceil$.  In the previous section, we described results of Zhai, Kaplan and Ye, Singhal, and Zhu, which all prove quantitative statements implying that the proportion of numerical semigroups that have depth larger than $3$ is not too large.  In our proof of Theorem \ref{Thm: Uniform conv psi_S} we pay particular attention to semigroups of depth $2$ and to semigroups of depth $3$.

We define the following two subsets of $\SS_g$ according to their depth:
\begin{eqnarray*}
\B(g) & = & \{S\in\SS_g \mid F(S)<2m(S)\},\\
\cC(g)& = & \{S\in\SS_g\mid 2m(S)<F(S)<3m(S)\}.
\end{eqnarray*}

We first focus on the subset of numerical semigroups of depth $2$ and recall a result of Zhao \cite{zhao2010constructing}.  We further partition the elements of $\B(g)$ by multiplicity.  Let
\[
\B(g,m)=\{S\in\SS_g\mid m(S)=m,F(S)<2m\}.
\]

If $A$ and $B$ are finite sets, we write $A+B = \{a+b\colon a \in A,\ b \in B\}$.  If $B$ is a set consisting of a single integer $b$, we will often write $A+b$ instead of $A+B$.

\begin{proposition}\cite[Corollary 2.2]{zhao2010constructing}\label{Zhao F<2m}
Numerical semigroups in $\B(g,m)$ are in bijection with subsets $B\subseteq \{1,2,\dots,m-1\}$ of size $2m-g-2$. The bijection is as follows.  Given such a subset $B$, let
\[
S_{m,B}=(m+B)\cup\{0,m,2m\rightarrow\} \in \B(g,m).
\]
\end{proposition}
\noindent Note that $\B(g,m)\neq\emptyset$ if and only if $\frac{g}{2}+1\leq m\leq g+1$.

We now describe a decomposition of the set of numerical semigroups in $\SS_g$ of depth $3$.  We first divide up the set of elements of $\cC(g)$ by multiplicity, and then further partition this subset with fixed multiplicity by the value of $F(S)-2m(S)$.  We define two sets:
\begin{eqnarray*}
\cC(g,m) & = & \{S\in\cC(g)\mid m(S)=m\},\\
\cC(g,m,l) & = & \{S\in\cC(g,m)\mid F(S)=2m(S)+l\}.
\end{eqnarray*}

Consider the map from $\cC(g,m,l)$ to $2^{\N}\times 2^{\N}$ defined by
\[
\chi_{g,m,l}(S) =\left(S\cap [m,m+l],S\cap [2m,2m+l]\right).
\]
We now describe how to decompose $\cC(g,m,l)$ in terms of $\chi_{g,m,l}$.
\begin{proposition}\label{Prop: parameterize C(g,m,l)}
Suppose $(A_1,A_2)$ is in the image of $\chi_{g,m,l}$. Then, the preimage of $(A_1,A_2)$ under $\chi_{g,m,l}$ is in bijection with subsets $B\subseteq \{1,2,\dots,m-l-1\}$ of size $2m-g+l-|A_1|-|A_2|$. The bijection is as follows.  Given such a subset $B$, let
\[S_B=\{0\}\cup A_1 \cup (m+l+ B)\cup A_2\cup\{2m+l+1\rightarrow\} \in \cC(g,m,l).\]
\end{proposition}
\begin{proof}
Since $(A_1,A_2)$ is in the image of $\chi_{g,m,l}$, there is some numerical semigroup $S$ for which $\chi_{g,m,l}(S)=(A_1,A_2)$, $g(S)=g$, $m(S)=m$ and $F(S)=2m+l$.

Now consider a subset $B\subseteq \{1,2,\dots,m-l-1\}$ of size $2m-g+l-|A_1|-|A_2|$. We want to show that $S_B$ is a numerical semigroup with $\chi_{g,m,l}(S_B)=(A_1,A_2)$.
It is clear that $m(S_B)=m$, $F(S_B)=2m+l$, and $g(S_B)=g$, so we only need to show that $S_B$ is closed under addition. 

Consider $x,y\in S_B\setminus\{0\}$. 
Assume for the sake of contradiction that $x+y\notin S_B$.
Since $x,y\geq m$, we know that $x+y\geq 2m$. Since $x+y\notin S_B$, we know that $x+y\leq 2m+l$.

Notice that $x,y\geq m$ and $x+y\leq 2m+l$ imply $x,y\leq m+l$. Therefore, $x,y\in S_B\cap [m,m+l]=A_1$, and so $x,y\in S$.
This means $x+y\in S\cap [2m,2m+l]$, that is, $x+y\in A_2$. However, this contradicts the assumption that $x+y\notin S_B$.
We conclude that $x+y\in S_B$, hence $S_B$ is a numerical semigroup. This completes the proof.
\end{proof}
Zhao also considers the collection of semigroups $\cC(g,m,l)$ in \cite{zhao2010constructing}.  Proposition 3.3 of that paper gives a decomposition that is similar to the one in Proposition \ref{Prop: parameterize C(g,m,l)}, but we find it more convenient to work with the version given here.

\subsection{Ordering by Frobenius Number}

Backelin proved an asymptotic approximation for $N_{\Frob}(F)$ \cite{backelin1990number}.

\begin{theorem}\cite[Proposition 1]{backelin1990number}
There are constants $c_1, c_2 > 0$ such that
\begin{eqnarray*}
\lim_{\substack{F \rightarrow\infty \\ F \text{ odd}}} \frac{N_{\Frob}(F)}{\sqrt{2}^F} & = & c_1,\\
\lim_{\substack{F \rightarrow\infty \\ F \text{ even}}} \frac{N_{\Frob}(F)}{\sqrt{2}^F} & = & c_2.
\end{eqnarray*}
\end{theorem}

Backelin also proved that most numerical semigroups have multiplicity close to $F/2$.
\begin{theorem}\cite[Proposition 2]{backelin1990number}\label{thm: backelin m=F/2}
Given $\epsilon>0$, there is $M>0$ such that for every $F$, we have
\[\P_{\Frob,F}\Big[\Big|m(S)-\frac{F}{2}\Big|>M\Big]<\epsilon.\]
\end{theorem}
This is an analogue of Theorem~\ref{F 2m} for numerical semigroups ordered by Frobenius number.
This result implies that as $F$ grows to infinity most numerical semigroups with Frobenius number $F$ have depth $2$ or $3$.
In fact, Li proves an upper bound on the number of semigroups with Frobenius number $F$ and larger depth in \cite{li2023counting}.
\begin{theorem}\cite[Corollary 1.4]{li2023counting}
The number of numerical semigroups in $\S_{\Frob}(F)$ with depth $q$ is 
\[
\left\lfloor\frac{(q+1)^2}{4}\right\rfloor^{\frac{F}{2q-2} +o(F)}.
\]
\end{theorem}

In \cite{singhal2022distribution}, Singhal shows that most numerical semigroups in $\S_{\Frob}(F)$ have genus close to $\frac{3}{4}F$.  This is the analogue of Equation \eqref{eqn:FG} following Theorem \ref{Kaplan Mult} for numerical semigroups ordered by Frobenius number.

\begin{theorem}\cite[Theorem 3]{singhal2022distribution}\label{thm: singhal genus by frob}
Given $\epsilon>0$, we have
\[
\lim_{F\to\infty} \P_{\Frob,F}\left[\left|g(S)-\frac{3}{4}F\right|<\epsilon F \right]=1.
\]
\end{theorem}

We define the following two subsets of $\SS_{\Frob}(F)$ according to their depth:
\begin{eqnarray*}
\B_{\Frob}(F) & = & \{S\in\SS_{\Frob}(F) \mid F<2m(S)\},\\
\cC_{\Frob}(F)& = & \{S\in\SS_{\Frob}(F)\mid 2m(S)<F<3m(S)\}.
\end{eqnarray*}

We first focus on the subset of numerical semigroups of depth $2$. We further partition $\B_{\Frob}(F)$ by genus. Let
\[
\B_{\Frob}(F,g)=\{S\in\SS_{\Frob}(F)\mid g(S)=g, F<2m(S)\}.
\]

\begin{proposition}\cite[Theorem 10]{singhal2022distribution}\label{Prop: depth 2 by frob}
Numerical semigroups in $\B_{\Frob}(F,g)$ are in bijection with subsets $B\subseteq \{1,2,\dots,\lfloor\frac{F-1}{2} \rfloor\}$ of size $F-g$. The bijection is as follows. Given such a subset $B$, let
\[
S_{F,B}=
\left( \left\lfloor \frac{F}{2} \right\rfloor+B\right) \cup\{0,F+1\rightarrow\} \in \B_{\Frob}(F).
\]
\end{proposition}

We now describe a decomposition of the set of numerical semigroups in $\SS_{\Frob}(F)$ of genus $g$ and depth $3$.
We first partition the elements of $\cC_{\Frob}(F)$ by genus and then by multiplicity. Let
\begin{align*}
\cC_{\Frob}(F,g)& =\{S\in\SS_{\Frob}(F)\mid g(S)=g,2m(S)<F<3m(S)\},\\
\cC_{\Frob}(F,g,m) & =\{S\in\SS_{\Frob}(F)\mid g(S)=g,m(S)=m,2m<F<3m\}.
\end{align*}

Note that $\cC_{\Frob}(F,g,m)=\cC(g,m,F-2m)$. Therefore, we can restate Proposition~\ref{Prop: parameterize C(g,m,l)} as follows.
\begin{proposition}\label{Prop: parameterize C(F,m,g) by frob}
Suppose $(A_1,A_2)$ is in the image of $\chi_{g,m,F-2m}$. Then, the preimage of $(A_1,A_2)$ under $\chi_{g,m,F-2m}$ is in bijection with subsets $B\subseteq \{1,2,\dots,3m-F-1\}$ of size $F-g-|A_1|-|A_2|$. The bijection is as follows.  Given such a subset $B$, let
\[S_B=\{0\}\cup A_1 \cup ((F-m)+ B)\cup A_2\cup\{F+1\rightarrow\} \in \cC_{\Frob}(F,g,m).\]
\end{proposition}

\section{From pointwise to uniform convergence}\label{sec:uniform_pointwise}

The goal of this section is to explain how Theorem~\ref{Thm: Uniform conv psi_S} follows from a statement about the pointwise convergence in probability as $g \rightarrow \infty$ of the functions $\psi_S(\alpha)$ to the function $\psi(\alpha)$.  Before getting into the proof, we review some background and tools from probability theory that we will apply throughout the rest of the paper.

\subsection{Probability background}\label{sec:prob}

Recall that for each genus $g$ we select a numerical semigroup $S$ uniformly at random from $\SS_g$, so all random quantities are functions of $S$ with respect to $(\SS_g,\P_g)$.  If $X$ is a random variable on $\SS_g$, we denote its expectation by $\E_{g}[X]$ and its variance by $\Var_{g}[X]$.  A sequence of events $E_g\subseteq \SS_g$ is said to occur asymptotically almost surely when
\[
\lim_{g\to\infty}\P_g[E_g]=1,
\]
equivalently, when $\P_g[E_g^c]\to 0$. Here $E_g^c$ is the complement of the event $E_g$ and by $\P_g[E_g^c]\to 0$, we mean the probability converges to $0$ as $g\to\infty$.

Given real-valued random variables $X_g$ on $(\SS_g,\P_g)$ and a constant $x\in\R$, we say that $X_g$ converges to $x$ in probability if for every $\epsilon>0$,
\[
\lim_{g\to\infty}\P_g\left(|X_g-x|>\epsilon\right)=0.
\]

We will consider random functions $\alpha\mapsto Y_g(\alpha)$ on $[0,1]$ and compare them with a deterministic function $y(\alpha)$. Pointwise convergence in probability means that for each fixed $\alpha\in[0,1]$ and each $\epsilon>0$,
\[
\lim_{g\to\infty}\P_g\left(|Y_g(\alpha)-y(\alpha)|>\epsilon\right)=0.
\]
Uniform convergence in probability strengthens this by requiring that for every $\epsilon>0$,
\[
\lim_{g\to\infty}\P_g\Big(\sup_{\alpha\in[0,1]}|Y_g(\alpha)-y(\alpha)|>\epsilon\Big)=0.
\]
In our application $Y_g(\alpha)=\psi_S(\alpha)$ and $y(\alpha)=\psi(\alpha)$.

We frequently use Boole’s inequality (the union bound): for any events $E_1,\dots,E_n$,
\[
\P_g\left(\bigcup_{i=1}^n E_i\right)\le \sum_{i=1}^n \P_g(E_i).
\]

We also use Chebyshev’s inequality. If $X$ is a random variable with finite variance, then for any $\epsilon>0$,
\[
\P_g\left(|X-\E_g[X]|\ge \epsilon\right)\le \frac{\Var_g(X)}{\epsilon^2}.
\]
A consequence useful to us is the following second-moment criterion: if $X_g$ are uniformly square-integrable and for some constant $\mu$ we have $\E_g[X_g]\to\mu$ and $\E_g[X_g^2]\to \mu^2$, then
\[
\Var_g(X_g)=\E_g[X_g^2]-\E_g[X_g]^2\ \longrightarrow\ 0,
\]
and therefore Chebyshev's inequality implies that $X_g\to \mu$ in probability.

Finally, we will repeatedly average conditional bounds using the law of total expectation. If $Z$ is a discrete random quantity (for example $Z=m(S)$), then
\[
\E_g[X]=\E_g\big(\E_g[X|Z]\big)=\sum_z \E_g[X|Z=z]\P_g(Z=z).
\]

\subsection{Proof of Theorem~\ref{Thm: Uniform conv psi_S} assuming pointwise convergence} 

We state the result about pointwise convergence here and defer the proof to Section~\ref{Sec: pointwise}.

\begin{theorem}\label{Thm: pointwise convergence}
Given $\alpha\in[0,1]\setminus\{\frac{1}{\sqrt{5}}\}$ and $\epsilon>0$, we have
\[\lim_{g\to\infty} \P_g[ |\psi_S(\alpha) -\psi(\alpha)|<\epsilon ] =1.\]
\end{theorem}

We show that Theorem~\ref{Thm: pointwise convergence} implies Theorem~\ref{Thm: Uniform conv psi_S}.

\begin{proof}[Proof of Theorem~\ref{Thm: Uniform conv psi_S} assuming Theorem~\ref{Thm: pointwise convergence}]
Consider $\epsilon>0$.
Choose $N$ such that $\frac{1}{N}\phi <\frac{\epsilon}{2}$. For each $i$ satisfying $0 \le i \le N$ let $\alpha_i=\frac{i}{N}$. Note that $\psi(\alpha_{i+1})-\psi(\alpha_i)<\frac{\epsilon}{2}$.
Now given any $\alpha\in [0,1]$, there is some $i$ for which $\alpha_i\leq \alpha\leq \alpha_{i+1}$. Note that
\begin{eqnarray*}
\psi_S(\alpha)-\psi(\alpha) & \leq \psi_S(\alpha_{i+1})-\psi(\alpha_i) & < \psi_S(\alpha_{i+1})-\psi(\alpha_{i+1}) +\frac{\epsilon}{2},\\
\psi(\alpha)-\psi_S(\alpha) & \leq \psi(\alpha_{i+1})-\psi_S(\alpha_i)  &< \psi(\alpha_{i})-\psi_S(\alpha_{i}) +\frac{\epsilon}{2}.
\end{eqnarray*}
From this we can conclude that
\[\sup_{0\leq \alpha\leq 1} |\psi_S(\alpha)-\psi(\alpha)| \leq \max_{0\leq i\leq N} |\psi_S(\alpha_i)- \psi(\alpha_i)| +\frac{\epsilon}{2}.\]
This means that if there is some $\alpha\in[0,1]$ for which $|\psi_S(\alpha)-\psi(\alpha)|>\epsilon$, then there is some $0\leq i\leq N$ for which $|\psi_S(\alpha_i)-\psi(\alpha_i)|>\frac{\epsilon}{2}$. This implies that
\[\P_g\Big[ \sup_{0\leq \alpha\leq 1}|\psi_S(\alpha) -\psi(\alpha)| >\epsilon \Big] \leq \sum_{i=0}^N \P_g\Big[ |\psi_S(\alpha_i) -\psi(\alpha_i)| >\frac{\epsilon}{2} \Big].\]

For each $0\leq i\leq N$, Theorem~\ref{Thm: pointwise convergence} says that
\[
\lim_{g\to\infty} \P_g\Big[ |\psi_S(\alpha_i) -\psi(\alpha_i)| >\frac{\epsilon}{2} \Big] = 0.
\]
The result follows.
\end{proof}

Note that Theorem~\ref{Thm: Uniform conv psi_S} implies pointwise convergence at $\alpha=\frac{1}{\sqrt{5}}$ as well, despite the fact that this was not covered by Theorem~\ref{Thm: pointwise convergence}.

\section{Pointwise convergence: The proof of Theorem \ref{Thm: pointwise convergence}}\label{Sec: pointwise}

The goal of this section is to prove Theorem~\ref{Thm: pointwise convergence}, that is, show that for each $\alpha\in [0,1]\setminus\{\frac{1}{\sqrt{5}}\}$, we have convergence in probability of $\psi_S(\alpha)$ to $\psi(\alpha)$.

We start with the cases where $\alpha=0$ and where $\frac{1}{\sqrt{5}}<\alpha\leq 1$, because these are easier to handle.

\begin{proposition}\label{Prop: pointwise alpha 0}
Let $\epsilon>0$.  We have
\[
\lim_{g\to\infty} \P_g[ |\psi_S(0) -\psi(0)|<\epsilon ] =1.
\]
\end{proposition}
\begin{proof}
For any $S \in \SS_g$ we have $\psi_S(0) = \frac{a_{1}(S)}{g} = \frac{m(S)}{g}$.  Therefore, Theorem~\ref{Kaplan Mult} implies that
\[
\lim_{g\to\infty} \P_g[ |\psi_S(0) -\psi(0)|<\epsilon ]
=\lim_{g\to\infty} \P_g[ |m(S) -\gamma g|<\epsilon g ]=1. \qedhere
\]
\end{proof}

The case where $\frac{1}{\sqrt{5}}<\alpha\leq 1$ is similar.
\begin{lemma}\label{Lem: t>F+1-g}
If $t\geq F(S)+1-g(S)$, then $a_t(S)=t+g(S)$.    
\end{lemma}
\begin{proof}
Note that $t+g(S)\geq F(S)+1$, therefore $t+g(S)\in S$ and
\[|S\cap [1,t+g(S)]| =(t+g(S))- g(S)=t.\]
The result follows.
\end{proof}

\begin{proposition}\label{Prop: pointwise large alpha}
Suppose $\alpha$ satisfies $\frac{1}{\sqrt{5}}<\alpha\leq 1$.  Let $\epsilon>0$.  We have
\[
\lim_{g\to\infty} \P_g[ |\psi_S(\alpha) -\psi(\alpha)|<\epsilon ] =1.
\]
\end{proposition}
\begin{proof}
Let $t=1+\lfloor\alpha (g-1)\rfloor$, so $\psi_S(\alpha)= \frac{a_{t}(S)}{g}$. If $F(S)\leq t+g-1$, then  Lemma~\ref{Lem: t>F+1-g} implies that $a_t(S) = t+g = 1+\lfloor\alpha (g-1)\rfloor + g$.  In this case,
\begin{align*}
|\psi_S(\alpha) -\psi(\alpha)|
&=\Big|\frac{a_t(S)}{g} - (1+\alpha)\Big|
= \Big|\frac{1+\lfloor\alpha (g-1)\rfloor +g}{g} - (1+\alpha)\Big|\\
&\leq \Big|\frac{1+\alpha (g-1) +g}{g} - (1+\alpha)\Big| +\frac{1}{g} 
<\frac{2}{g}.
\end{align*}

Therefore, for $g>\frac{2}{\epsilon}$, we see that
\[
\P_g[ |\psi_S(\alpha) -\psi(\alpha)|<\epsilon ] 
\ge \P_g[ F(S) \le  g+\lfloor\alpha (g-1)\rfloor]
\geq \P_g[ F(S) \le  g(1+\alpha) -2].
\]

Since $\alpha>\frac{1}{\sqrt{5}}$, we know that $1+\alpha>2\gamma$. The result from \eqref{eqn:FG} following Theorem~\ref{Kaplan Mult} implies that
\[
\lim_{g\to\infty} \P_g[ F(S)<  g(1+\alpha) -2] =1.
\]
The result follows. 
\end{proof}

For the case where $0<\alpha<\frac{1}{\sqrt{5}}$, we will show that 
\[
\lim_{g\to\infty} \E_g[\psi_S(\alpha)] =\psi(\alpha)\ \ \ \text{ and }\ \ \  \lim_{g\to\infty} \Var_g[\psi_S(\alpha)] =0.
\] 
As we explained at the end of Section \ref{sec:prob}, Chebyshev's inequality then implies the pointwise convergence result that we want.  For these computations we consider numerical semigroups of depth $2$ and numerical semigroups of depth $3$ separately. Since most numerical semigroups have depth $2$ or $3$, this is enough.

As seen in Proposition~\ref{Zhao F<2m} and Proposition~\ref{Prop: parameterize C(g,m,l)}, numerical semigroups of depth $2$ and $3$ are obtained from random subsets of  given size chosen from a given interval. We want to understand the statistics of the $k^{th}$ element of such a random subset. We will use the following result about order statistics.
\begin{theorem}\cite[Theorem 2]{o2022distribution}\label{Thm: order statistic}
Suppose a random subset of size $n$ is chosen from $1,\dots,N$. Order the elements of the subset as $X_{(1)}<X_{(2)}<\dots<X_{(n)}$.
Then we have
\begin{align*}
\E[X_{(k)}] &=\frac{N+1}{n+1}k,
&
\Var[X_{(k)}] =\frac{(N+1)(N-n)}{(n+1)^2 (n+2)} k (n+1-k).
\end{align*}
\end{theorem}

\subsection{Depth 2, $0<\alpha<\frac{1}{\sqrt{5}}$}\label{Subsec: depth 2}

By Proposition~\ref{Zhao F<2m}, if $S\in \B(g,m)$, then $S=S_{m,B}$ for some subset $B\subseteq\{1,\dots,m-1\}$ of size $2m-g-2$.
Note that $a_{k+1}(S)$ corresponds to the $k^{th}$ smallest element of $B$. Therefore, Theorem~\ref{Thm: order statistic} directly implies the following.

\begin{corollary}\label{Cor: ak+1 depth 2}
For $1\leq k\leq 2m-g-2$, we have
\begin{eqnarray*}
\E_g[a_{k+1}(S)\mid S\in \B(g,m)] & = & m+\frac{mk}{2m-g-1},\\
\Var_g[a_{k+1}(S)\mid S\in \B(g,m)] & = & \frac{m(g-m+1) k (2m-g-1-k)}{(2m-g-1)^2 (2m-g)}.
\end{eqnarray*}
\end{corollary}

Instead of keeping track of $\Var_g[\psi_S(\alpha)]$, from now on, we will keep track of $\E_g[\psi_S(\alpha)^2]$. Theorem~\ref{Kaplan Mult} implies that most numerical semigroups in $\SS_g$ have $m(S)$ close to $\gamma g$.  We consider the semigroups for which this holds.  We show that the expectation of $\psi_S(\alpha)$ is close to $\psi(\alpha)$ and that the expectation of $\psi_S(\alpha)^2$ is close to $\psi(\alpha)^2$.

\begin{lemma}\label{Lem: exp, var depth 2}
Assume $\epsilon_1< \frac{1}{6\sqrt{5}}$ and $g>\frac{1}{\epsilon_1}$.
Suppose $m$ and $\alpha$ satisfy $|m-\gamma g|<\epsilon_1 g$ and $0< \alpha \leq \frac{1}{\sqrt{5}} -4\epsilon_1$.
Then
\[
\Bigl|\E_g \big[\psi_S(\alpha)\mid S\in\mathcal B(g,m)\big]-(\gamma+\phi\alpha)\Bigr|
\leq (8\phi-1)\epsilon_1. \]
Moreover,
\[
\Bigl|\E_g \big[\psi_S(\alpha)^2\mid S\in\mathcal B(g,m)\big]-(\gamma+\phi\alpha)^2\Bigr|
\leq 68\epsilon_1. \]
\end{lemma}
\begin{proof}
Set $k=\lfloor \alpha (g-1)\rfloor$, so $\psi_S(\alpha)= \frac{a_{k+1}(S)}{g}$.  By assumption, $2m-g \ge (2\gamma-1)g - 2\epsilon_1 g$.  Since $2\gamma-1 = \frac{1}{\sqrt{5}}$, the bound on $\alpha$ ensures that $0\leq k\leq 2m-g-2$. 
Corollary~\ref{Cor: ak+1 depth 2} implies that
\[
\E_g[a_{k+1}(S)\mid S\in \B(g,m)] = m+\frac{mk}{2m-g-1}.
\]
Write $\rho=m/g$ and $\kappa=k/g$.  By hypothesis $|\rho-\gamma|<\epsilon_1$ and $\alpha-\frac{2}{g} <\kappa<\alpha$.
Dividing by $g$, we see that
\[
\E_g[\psi_S(\alpha)\mid S\in \B(g,m)]
= \rho+\frac{\rho\kappa}{2\rho-1-\frac{1}{g}}.
\]
Now, notice that
\begin{align*}
\Bigl|\E_g \big[\psi_S(\alpha)\mid S\in\mathcal B(g,m)\big]-(\gamma+\phi\alpha)\Bigr|
&= \Bigl| \rho+\frac{\rho \kappa}{2\rho-1-\frac{1}{g}} -(\gamma+\phi\alpha)\Bigr|\\
&\leq |\rho-\gamma| + \Bigl| \frac{\rho \kappa}{2\rho-1-\frac{1}{g}} -\phi\alpha\Bigr|\\
&<\epsilon_1 + \kappa\Bigl| \frac{\rho}{2\rho-1-\frac{1}{g}} -\phi\Bigr| 
+ \phi | \kappa - \alpha |.
\end{align*}
Now, since $\frac{1}{g}<\epsilon_1$ and $\phi=\sqrt{5}\gamma$, we have
\[
\left|\rho -\phi \left(2\rho-1-\frac{1}{g}\right)\right| < |\phi -\rho (2\phi-1)| + 
\frac{\phi}{g} 
< \sqrt{5}\left|\gamma -\rho \right|  +\phi\epsilon_1 < (\sqrt{5}+\phi)\epsilon_1.
\]
Since $\frac{1}{g}<\epsilon_1 < \frac{1}{6\sqrt{5}}$ and $2\gamma-1=\frac{1}{\sqrt{5}}$, we have
\[
2\rho-1-\frac{1}{g} > 2\gamma-1 -2\epsilon_1 -\frac{1}{g} >\frac{1}{\sqrt{5}} -3\epsilon_1 > \frac{1}{2\sqrt{5}}.
\]
Therefore, we see that
\begin{align*}
\epsilon_1 + \kappa\Bigl| \frac{\rho}{2\rho-1-\frac{1}{g}} -\phi\Bigr|  + \phi | \kappa - \alpha |
&< \epsilon_1 + \frac{1}{\sqrt{5}} \frac{(\sqrt{5}+\phi)\epsilon_1}{\frac{1}{2\sqrt{5}}} + \phi 2\epsilon_1
=(8\phi-1) \epsilon_1.
\end{align*}

Next, we consider the variance.
Again from Corollary~\ref{Cor: ak+1 depth 2}, we know that
\[\Var_g[a_{k+1}(S)\mid S\in \B(g,m)]
=\frac{m(g-m+1)k(2m-g-1-k)}{(2m-g-1)^2(2m-g)}.\]
Divide by $g^2$ to pass to $\psi_S(\alpha)= \frac{a_{k+1}(S)}{g}$:
\[
\Var_g[\psi_S(\alpha)\mid S\in \B(g,m)]
=\frac{\rho(1-\rho+\tfrac1g) \kappa (2\rho-1-\kappa-\tfrac1g)}
{g (2\rho-1-\tfrac1g)^2(2\rho-1)}.
\]
Therefore,
\begin{align*}
\Var_g[\psi_S(\alpha)\mid S\in \B(g,m)]
&<\frac{1}{g} \frac{\rho \kappa (2\rho-1)}
{ \big(\frac{1}{2\sqrt{5}}\big)^3}
<\frac{1}{g} \left(2\sqrt{5}\right)^3 (\gamma+\epsilon_1) \alpha (2\gamma+2\epsilon_1-1)\\
&<\frac{1}{g} \left(2\sqrt{5}\right)^3 \left(\gamma+\frac{1}{6\sqrt{5}}\right) \frac{1}{\sqrt{5}} \left(2\gamma+\frac{1}{3\sqrt{5}}-1\right)
<\frac{20}{g} < 20\epsilon_1.
\end{align*}
Now,
\begin{align*}
& \ \Bigl|\E_g \big[\psi_S(\alpha)^2\mid S\in\mathcal B(g,m)\big]-(\gamma+\phi\alpha)^2\Bigr|\\
=&\   \Bigl| \Var_g[\psi_S(\alpha)\mid S\in \B(g,m)] +\E_g \big[\psi_S(\alpha)\mid S\in\mathcal B(g,m)\big]^2 -(\gamma+\phi\alpha)^2\Bigr|\\
\leq &\ \Bigl| \Var_g[\psi_S(\alpha)\mid S\in \B(g,m)] \Bigr|\\
& +\Bigl|\E_g \big[\psi_S(\alpha)\mid S\in\mathcal B(g,m)\big] -(\gamma+\phi\alpha)\Bigr|
\cdot \Bigl|\E_g \big[\psi_S(\alpha)\mid S\in\mathcal B(g,m)\big] +(\gamma+\phi\alpha)\Bigr|\\
<& \  20\epsilon_1 + (8\phi-1)\epsilon_1 (2+2)
<68\epsilon_1.\qedhere
\end{align*}
\end{proof}

We use the law of total expectation to combine our results for the expectations taken over semigroups $S\in \B(g,m)$ to derive a result about the expectation taken over the whole collection of $S \in \B(g)$.

\begin{proposition}\label{Prop: small alpha depth 2}
Suppose that $0<\alpha<\frac{1}{\sqrt{5}}$.
Then we have
\[\lim_{g\to\infty} \E_g \big[\psi_S(\alpha)\mid S\in\mathcal B(g)\big] = \gamma+\phi\alpha. \]
Moreover,
\[
\lim_{g\to\infty} \E_g \big[\psi_S(\alpha)^2\mid S\in\mathcal B(g)\big] = (\gamma+\phi\alpha)^2. 
\]
\end{proposition}
\begin{proof}
Pick $0<\epsilon_1 < \min\left(\frac{1}{6\sqrt{5}}, \frac{1}{4} (\frac{1}{\sqrt{5}}-\alpha ) \right)$ and $g>\frac{1}{\epsilon_1}$. Recall that $0\leq \psi_S(\alpha)\leq 2$.
By the law of total expectation and Lemma~\ref{Lem: exp, var depth 2}, we have
\begin{align*}
& \ \Big|\E_g \big[\psi_S(\alpha)\mid S\in\mathcal B(g)\big] -(\gamma+\phi\alpha) \Big|\\
\leq &\ \sum_{m: |m-\gamma g|<\epsilon_1 g} \Big|\E_g \big[\psi_S(\alpha)\mid S\in\mathcal B(g,m)\big] -(\gamma+\phi\alpha) \Big| \P[m(S)=m\mid S\in\mathcal B(g)]\\
& +\sum_{m: |m-\gamma g|\ge \epsilon_1 g} \Big|\E_g \big[\psi_S(\alpha)\mid S\in\mathcal B(g,m)\big] -(\gamma+\phi\alpha) \Big| \P[m(S)=m\mid S\in\mathcal B(g)]\\
\leq &\  (8\phi-1)\epsilon_1 \P[|m(S)-\gamma g|<\epsilon_1 g\mid S\in\mathcal B(g)]
+ 2 \P[|m(S)-\gamma g| \ge \epsilon_1 g\mid S\in\mathcal B(g)].
\end{align*}
Theorem~\ref{Kaplan Mult} implies that
\[
\limsup_{g\to\infty} \Big|\E_g \big[\psi_S(\alpha)\mid S\in\mathcal B(g)\big] -(\gamma+\phi\alpha) \Big| \leq (8\phi-1)\epsilon_1.
\]
Since $\epsilon_1$ was arbitrary, we see that
\[
\lim_{g\to\infty} \E_g \big[\psi_S(\alpha)\mid S\in\mathcal B(g)\big] =(\gamma+\phi\alpha).
\]

Similarly, we also see that
\begin{align*}
& \ \Big|\E_g \big[\psi_S(\alpha)^2\mid S\in\mathcal B(g)\big] -(\gamma+\phi\alpha)^2 \Big|\\
\leq &\ \sum_{m: |m-\gamma g|<\epsilon_1 g} \Big|\E_g \big[\psi_S(\alpha)^2\mid S\in\mathcal B(g,m)\big] -(\gamma+\phi\alpha)^2 \Big| \P[m(S)=m\mid S\in\mathcal B(g)]\\
& +\sum_{m: |m-\gamma g| \ge \epsilon_1 g} \Big|\E_g \big[\psi_S(\alpha)^2\mid S\in\mathcal B(g,m)\big] -(\gamma+\phi\alpha)^2 \Big| \P[m(S)=m\mid S\in\mathcal B(g)]\\
\leq &\ 68\epsilon_1 \P[|m(S)-\gamma g|<\epsilon_1 g\mid S\in\mathcal B(g)]
+ 4 \P[|m(S)-\gamma g|\ge \epsilon_1 g\mid S\in\mathcal B(g)].
\end{align*}
Theorem~\ref{Kaplan Mult} implies that
\[
\limsup_{g\to\infty} \Big|\E_g \big[\psi_S(\alpha)^2\mid S\in\mathcal B(g)\big] -(\gamma+\phi\alpha)^2 \Big| \leq 68\epsilon_1.
\]
Since $\epsilon_1$ was arbitrary, we see that
\[
\lim_{g\to\infty} \E_g \big[\psi_S(\alpha)^2\mid S\in\mathcal B(g)\big] =(\gamma+\phi\alpha)^2.\qedhere
\]
\end{proof}

\subsection{Depth 3, $0<\alpha<\frac{1}{\sqrt{5}}$}

This subsection is analogous to Subsection~\ref{Subsec: depth 2} and deals with semigroups $S \in \cC(g)$. We start by applying Theorem~\ref{Thm: order statistic} to the setting of Proposition~\ref{Prop: parameterize C(g,m,l)}.  

\begin{lemma}\label{Lem: aks depth 3}
Consider $(A_1,A_2)\in \Ima(\chi_{g,m,l})$. For $|A_1|<k\leq 2m-g+l-|A_2|$, we have
\[
\E_g[a_{k}(S)\mid S\in \cC(g,m,l), \chi_{g,m,l}(S)=(A_1,A_2)] =m+l+\frac{(m-l)(k-|A_1|)}{2m-g+l-|A_1|-|A_2|+1},
\]
and 
\begin{align*}
&\Var_g[a_{k}(S)\mid S\in \cC(g,m,l), \chi_{g,m,l}(S)=(A_1,A_2)] \\
&=\frac{(m-l)(g-m-2l-1+|A_1|+|A_2|) (k-|A_1|) (2m-g+l-|A_2|+1-k)}{(2m-g+l-|A_1|-|A_2|+1)^2 (2m-g+l-|A_1|-|A_2|+2)}.  
\end{align*}
\end{lemma}
\begin{proof}
By Proposition~\ref{Prop: parameterize C(g,m,l)}, we know that numerical semigroups with $\chi_{g,m,l}(S) = (A_1,A_2)$ are in bijection with  $B\subseteq \{1,2,\dots,m-l-1\}$ of size $2m-g+l-|A_1|-|A_2|$. Also note that $a_k(S)$ corresponds to the $(k-|A_1|)^{th}$ element of $B$. The result follows from Theorem~\ref{Thm: order statistic}.
\end{proof}

Since most numerical semigroups have $|m(S)-\gamma g|<\epsilon g $ and $|F(S)-2m(S)|<\epsilon g $, we focus on these semigroups.

\begin{lemma}\label{Lem: depth 3 fix A1 A2}
Assume $\epsilon_1< \frac{1}{6\sqrt{5}}$ and $g>\frac{1}{\epsilon_1}$.
Suppose $g,m,l$ satisfy $|m-\gamma g|<\epsilon_1 g$ and $l<\epsilon_1 g$.  Also suppose that $(A_1,A_2)\in \Ima(\chi_{g,m,l})$ and $\alpha$ satisfies $3\epsilon_1< \alpha \leq \frac{1}{\sqrt{5}} -2\epsilon_1$.
Then
\[
\Bigl|\E_g \big[\psi_S(\alpha)\mid S\in \cC(g,m,l), \chi_{g,m,l}(S)=(A_1,A_2) \big]-(\gamma+\phi\alpha)\Bigr|
\leq 25\epsilon_1. 
\]
Moreover,
\[
\Bigl|\E_g \big[\psi_S(\alpha)^2\mid S\in \cC(g,m,l), \chi_{g,m,l}(S)=(A_1,A_2) \big]-(\gamma+\phi\alpha)^2\Bigr|
< 109\epsilon_1. 
\]
\end{lemma}
\begin{proof}
Let $t=1+\lfloor\alpha(g-1)\rfloor$, so $\psi_S(\alpha)= \frac{a_{t}(S)}{g}$. The condition $3\epsilon_1< \alpha \leq \frac{1}{\sqrt{5}} -2\epsilon_1$, along with $|m-\gamma g|<\epsilon_1 g$ and $l<\epsilon_1 g$ ensure that $|A_1|<t\leq 2m-g+l-|A_2|$.

Write the scaled quantities
\[
\rho=\frac{m}{g},\qquad \lambda=\frac{l}{g},\qquad
a=\frac{|A_1|}{g},\qquad b=\frac{|A_2|}{g},\qquad
\kappa=\frac{t}{g},
\]
and set
\[D=2\rho-1+\lambda-a-b+\frac{1}{g}.\]

First, since $|m-\gamma g|\leq \epsilon_1 g$, $l\leq \epsilon_1 g$, $|A_1|,|A_2|\leq l$, and $\frac{1}{g} \leq \epsilon_1$, we have
\[
|\rho-\gamma|\leq \epsilon_1,\qquad
0 < \lambda\leq \epsilon_1,\qquad
0\leq a,b\leq \lambda.
\]
Also $|\kappa-\alpha|\leq \frac{1}{g}\leq \epsilon_1$ since $t=1+\lfloor\alpha(g-1)\rfloor$.

We apply the result of Lemma~\ref{Lem: aks depth 3} not to the function $a_k(S)$ but to the function $\frac{a_k(S)}{g}$.  Dividing by this constant $g$ divides the expected value by $g$ and the variance by $g^2$.  We see that
\[
\E_g\left[\psi_S(\alpha)\mid S\in \cC(g,m,l),\ \chi_{g,m,l}(S)=(A_1,A_2)\right]
= \rho+\lambda+\frac{(\rho-\lambda)(\kappa-a)}{D}.
\]
Hence
\begin{align*}
& \Bigl| \E_g\left[\psi_S(\alpha)\mid S\in \cC(g,m,l),\ \chi_{g,m,l}(S)=(A_1,A_2)\right]-(\gamma+\phi\alpha)\Bigr| \\
\leq & |\rho-\gamma|+\lambda+\phi|\kappa-\alpha-a|
   +\Bigl|\frac{\rho-\lambda}{D}-\phi\Bigr|\cdot|\kappa-a|.
\end{align*}
We bound each of the terms in this expression.
\begin{enumerate}
    \item We know that $|\rho-\gamma|\leq \epsilon_1$ and $\lambda\leq \epsilon_1$.
    \item We know that $|\kappa-\alpha|\leq \epsilon_1$ and $a\leq \epsilon_1$, so
    \[
    \phi|\kappa-\alpha-a|\leq 2\phi\epsilon_1.
    \]
    \item We now consider the term $D$.  Since $a+b \le 2\lambda$ and $2\rho \ge 2\gamma - 2\epsilon_1$, we have
    \[
    D\geq (2\rho-1)+\lambda-2\lambda\geq (2\gamma-1)-2\epsilon_1-\epsilon_1
    = \frac1{\sqrt5}-3\epsilon_1 \ \ge\ \frac1{2\sqrt5},
    \]
    because $\epsilon_1<\frac{1}{6\sqrt5}$. Hence $D^{-1}\leq 2\sqrt5$.
    \item Next,
    \begin{align*}
    \Bigl|\frac{\rho-\lambda}{D}-\phi\Bigr|
    &=\frac{\big|(\rho-\lambda)-\phi(2\rho-1+\lambda-a-b+\tfrac1g)\big|}{D}\\
    &=\frac{\big|(1-2\phi)(\rho-\gamma)-(1+\phi)\lambda+\phi(a+b)-\tfrac{\phi}{g}\big|}{D}\\
    &\leq \frac{(2\phi-1)\epsilon_1 +(1+\phi)\epsilon_1 +2\phi\epsilon_1+\phi\epsilon_1}{D}
    = \frac{6\phi\epsilon_1}{D}.
    \end{align*}

Applying the bound $D^{-1}\leq 2\sqrt5$ from (3), we see that
    \[
    \Bigl|\frac{\rho-\lambda}{D}-\phi\Bigr|\leq 12\phi\sqrt{5}\epsilon_1.
    \]
    Moreover,
    \[
    |\kappa-a|\leq \kappa +a 
    \leq \left(\frac{1}{g}+\alpha\right) +\epsilon_1 
    \leq \left(\epsilon_1 +\frac1{\sqrt5}-2\epsilon_1\right)+\epsilon_1
    = \frac1{\sqrt5}.
    \]
    This implies
    \[    
    \Bigl|\frac{\rho-\lambda}{D}-\phi\Bigr|\cdot|\kappa-a|
    \leq 12\phi\epsilon_1.
    \]
\end{enumerate}
Collecting (1)-(4), we see that
\begin{align*}
& \Bigl| \E_g\left[\psi_S(\alpha)\mid S\in \cC(g,m,l),\ \chi_{g,m,l}(S)=(A_1,A_2)\right]-(\gamma+\phi\alpha)\Bigr| \\
 \leq & \epsilon_1+\epsilon_1+2\phi\epsilon_1+12\phi \epsilon_1
= (2+14\phi)\epsilon_1 <25\epsilon_1.
\end{align*}
This proves the claimed bound for the expectation.

Dividing the formula for the variance given in Lemma~\ref{Lem: aks depth 3} by $g^2$ gives
\begin{align*}
&\Var_g[\psi_S(\alpha)\mid S\in\cC(g,m,l),\ \chi_{g,m,l}(S)=(A_1,A_2)]\\
&=\frac{(\rho-\lambda)\bigl(1-\rho-2\lambda-\tfrac1g+a+b\bigr)
(\kappa-a)\bigl(2\rho-1+\lambda-b+\tfrac1g-\kappa\bigr)}
{gD^2(D+\tfrac1g)}.    
\end{align*}
We bound the numerator and denominator using the hypotheses.

Since $D\geq \frac{1}{2\sqrt{5}}$, we know that
\[D^2(D+\tfrac1g)\ \ge\ \Big(\frac1{2\sqrt5}\Big)^3.\]

We bound each factor of the numerator from above:
\[
\rho-\lambda\ \le\ \gamma+\epsilon_1,\qquad
1-\rho-2\lambda-\tfrac1g+a+b\ \le\ 1-\rho\ \le\ 1-\gamma+\epsilon_1,
\]
\[
\kappa-a\ \le\ \alpha+\tfrac1g\ \le\ \frac1{\sqrt5},\qquad
2\rho-1+\lambda-b+\tfrac1g-\kappa\ \le\ 2\rho-1+\lambda+\tfrac1g
\ \le\ \frac1{\sqrt5}+4\epsilon_1.
\]
Therefore
\begin{align*}
&\Var_g[\psi_S(\alpha)\mid S\in\cC(g,m,l),\ \chi_{g,m,l}(S)=(A_1,A_2)]\\
\le &\ \frac{1}{g}(2\sqrt5)^3
(\gamma+\epsilon_1)(1-\gamma+\epsilon_1)\Big(\frac1{\sqrt5}\Big)
\Big(\frac1{\sqrt5}+4\epsilon_1\Big).
\end{align*}
We evaluate the right-hand side at the largest $\epsilon_1$ allowed by our hypotheses, 
$\epsilon_1=\frac{1}{6\sqrt5}$, and see that
\[
\Var_g[\psi_S(\alpha)\mid S\in\cC(g,m,l),\ \chi_{g,m,l}(S)=(A_1,A_2)]<\frac{9}{g}< 9\epsilon_1.
\]

Now,
\begin{align*}
&\  \Bigl|\E_g \big[\psi_S(\alpha)^2\mid S\in\cC(g,m,l),\ \chi_{g,m,l}(S)=(A_1,A_2)\big]-(\gamma+\phi\alpha)^2\Bigr|\\
= &\  \Bigl| \Var_g[\psi_S(\alpha)\mid \cC(g,m,l),\ \chi_{g,m,l}] +\E_g \big[\psi_S(\alpha)\mid \cC(g,m,l),\ \chi_{g,m,l}\big]^2 -(\gamma+\phi\alpha)^2\Bigr|\\
\leq &\ \Bigl| \Var_g[\psi_S(\alpha)\mid \cC(g,m,l),\ \chi_{g,m,l}] \Bigr|\\
 &+\Bigl|\E_g \big[\psi_S(\alpha)\mid \cC(g,m,l),\ \chi_{g,m,l}\big] -(\gamma+\phi\alpha)\Bigr|
\cdot \Bigl|\E_g \big[\psi_S(\alpha)\mid \cC(g,m,l),\ \chi_{g,m,l}\big] +(\gamma+\phi\alpha)\Bigr|\\
< &\ 9\epsilon_1 + 25\epsilon_1 (2+2)
=109\epsilon_1.\qedhere
\end{align*}
\end{proof}

We use the law of total expectation and average over different $(A_1,A_2)$ in the image of $\chi_{g,m,l}$ in order to obtain the expectations conditional on $S\in\cC(g,m,l)$.

\begin{corollary}\label{Cor: depth 3 fix gml}
Assume $\epsilon_1< \frac{1}{6\sqrt{5}}$ and $g>\frac{1}{\epsilon_1}$.
Suppose we are given $g,m$ and $l$ that satisfy $|m-\gamma g|<\epsilon_1 g$ and $l<\epsilon_1 g$, and that $\alpha$ satisfies $3\epsilon_1< \alpha \leq \frac{1}{\sqrt{5}} -2\epsilon_1$.
Then
\[
\Bigl|\E_g \big[\psi_S(\alpha)\mid S\in \cC(g,m,l) \big]-(\gamma+\phi\alpha)\Bigr|
\leq 25\epsilon_1. 
\]
Moreover,
\[
\Bigl|\E_g \big[\psi_S(\alpha)^2\mid S\in \cC(g,m,l) \big]-(\gamma+\phi\alpha)^2\Bigr|
\leq 109\epsilon_1. 
\]
\end{corollary}
\begin{proof}
By the law of total expectation, we can average over the $(A_1,A_2)$ in the image of $\chi_{g,m,l}$, that is
\[
\E_g\big[\psi_S(\alpha)\mid S\in \cC(g,m,l)\big]
=\E_g\big[\E_g\big[\psi_S(\alpha)\mid S\in \cC(g,m,l),\chi_{g,m,l}\big]\big].
\]

From Lemma~\ref{Lem: depth 3 fix A1 A2}, we have
\[
\Bigl|\E_g \big[\psi_S(\alpha)\mid S\in \cC(g,m,l),\chi_{g,m,l}(S)=(A_1,A_2)\big]
-(\gamma+\phi\alpha)\Bigr|\leq 25\epsilon_1
\]
for every $(A_1,A_2)$ in the image. Averaging over $\chi_{g,m,l}$ preserves
this bound, proving the first claim. The proof of the second claim is identical.
\end{proof}

We again use the law of total expectation, this time averaging over different $m,l$ in order to obtain the expectations conditional on $S\in\cC(g)$.

\begin{proposition}\label{Prop: small alpha depth 3}
Suppose that $0<\alpha<\frac{1}{\sqrt{5}}$.
Then we have
\[
\lim_{g\to\infty} \E_g \big[\psi_S(\alpha)\mid S\in \cC(g)\big] = \gamma+\phi\alpha. 
\]
Moreover,
\[
\lim_{g\to\infty} \E_g \big[\psi_S(\alpha)^2\mid S\in \cC(g)\big] = (\gamma+\phi\alpha)^2. 
\]
\end{proposition}
\begin{proof}
Choose
\[
0<\epsilon_1<\min\left(\frac{1}{6\sqrt5},\ \frac{\alpha}{3},\ \frac{\frac{1}{\sqrt{5}}-\alpha}{2}\right),
\]
and take $g>1/\epsilon_1$. Consider the event
\[
G_{\epsilon_1}=\big\{|m(S)-\gamma g|<\epsilon_1 g\big\}\cap\big\{F(S)-2m(S)<\epsilon_1 g\big\}.
\]
By Theorem~\ref{Kaplan Mult}, we know that
$\lim_{g\to\infty} \P_g[G_{\epsilon_1}\mid S\in\cC(g)]=1$.

Decompose by $(m,l)$ inside $\cC(g)$:
\[
\E_g\big[\psi_S(\alpha)\mid S\in\cC(g)\big]
=\sum_{m,l}\E_g\big[\psi_S(\alpha)\mid S\in\cC(g,m,l)\big]
\P\big(S\in\cC(g,m,l)\mid S\in\cC(g)\big).
\]
For $(m,l)$ corresponding to $G_{\epsilon_1}$, Corollary~\ref{Cor: depth 3 fix gml} gives
\[
\Bigl|\E_g\big[\psi_S(\alpha)\mid S\in\cC(g,m,l)\big]-(\gamma+\phi\alpha)\Bigr|
\leq 25\epsilon_1.
\]
For $(m,l)$ corresponding to $G_{\epsilon_1}^c$, we use the trivial bound
\[
\big|\E_g[\psi_S(\alpha)\mid S\in\cC(g,m,l)]-(\gamma+\phi\alpha)\big|\leq 2,
\]
which holds because $\psi_S(\alpha)\in[0,2]$.
Therefore
\begin{align*}
\Big|\E_g[\psi_S(\alpha)\mid S\in\cC(g)]-(\gamma+\phi\alpha)\Big|
&\leq 25\epsilon_1\ \P(G_{\epsilon_1}\mid S\in\cC(g))
+2\P(G_{\epsilon_1}^c\mid S\in\cC(g)).
\end{align*}
Taking $\limsup_{g\to\infty}$ yields
\[\limsup_{g\to\infty}\Big|\E_g[\psi_S(\alpha)\mid S\in\cC(g)]-(\gamma+\phi\alpha)\Big|
\leq 25\epsilon_1.\]
Since $\epsilon_1$ can be arbitrarily small, we conclude
\[\lim_{g\to\infty}\E_g[\psi_S(\alpha)\mid S\in\cC(g)]=\gamma+\phi\alpha.\]

By a similar argument, we see that
\begin{align*}
\Big|\E_g[\psi_S(\alpha)^2\mid S\in\cC(g)]-(\gamma+\phi\alpha)^2\Big|
&\leq 109\epsilon_1\ \P(G_{\epsilon_1}\mid S\in\cC(g))
+4\P(G_{\epsilon_1}^c\mid S\in \cC(g)).
\end{align*}
Taking $\limsup_{g\to\infty}$ yields
\[
\limsup_{g\to\infty}\Big|\E_g[\psi_S(\alpha)^2\mid S\in\cC(g)]-(\gamma+\phi\alpha)^2\Big|
\leq 109\epsilon_1.
\]
Since $\epsilon_1$ can be arbitrarily small, we conclude
\[
\lim_{g\to\infty}\E_g[\psi_S(\alpha)^2\mid S\in\cC(g)]=(\gamma+\phi\alpha)^2.\qedhere
\]
\end{proof}

\subsection{Concluding $0<\alpha<\frac{1}{\sqrt{5}}$}

Next, we combine Proposition~\ref{Prop: small alpha depth 2} and Proposition~\ref{Prop: small alpha depth 3} to get the unconditional expectations.

\begin{proposition}\label{Prop: small alpha}
Suppose that $0<\alpha<\frac{1}{\sqrt{5}}$.
Then we have
\[\lim_{g\to\infty} \E_g \big[\psi_S(\alpha)\big] = \gamma+\phi\alpha. \]
Moreover,
\[
\lim_{g\to\infty} \E_g \big[\psi_S(\alpha)^2\big] = (\gamma+\phi\alpha)^2. 
\]
\end{proposition}
\begin{proof}
Notice that
\begin{align*}
& \lim_{g\to\infty} \E_g \big[\psi_S(\alpha)\big]\\ 
= &\ \lim_{g\to\infty} \P_g[S\in \B(g)] \E_g \big[\psi_S(\alpha) \mid S\in \B(g)\big]
+\P_g[S\in \cC(g)] \E_g \big[\psi_S(\alpha) \mid S\in \cC(g)\big]\\
& +\P_g[S\notin \B(g)\cup\cC(g)] \E_g \big[\psi_S(\alpha) \mid S\notin \B(g)\cup\cC(g)\big].
\end{align*}
For the last term, note that
\[
\limsup_{g\to\infty}\P_g[S\notin \B(g)\cup\cC(g)] \E_g \big[\psi_S(\alpha) \mid S\notin \B(g)\cup\cC(g)\big]
\leq \lim_{g\to\infty}\P_g[S\notin \B(g)\cup\cC(g)] 2 =0.
\]
Therefore,
\[
\lim_{g\to\infty} \E_g \big[\psi_S(\alpha)\big]
= (\gamma+\phi\alpha ) \left(\lim_{g\to\infty} \P_g[S\in \B(g)] 
+\P_g[S\in \cC(g)] \right)
=(\gamma+\phi\alpha ).
\]

Similarly, we have
\begin{align*}
&\lim_{g\to\infty} \E_g \big[\psi_S(\alpha)^2\big]\\ 
=&\ \lim_{g\to\infty} \P_g[S\in \B(g)] \E_g \big[\psi_S(\alpha)^2 \mid S\in \B(g)\big]
+\P_g[S\in \cC(g)] \E_g \big[\psi_S(\alpha)^2 \mid S\in \cC(g)\big]\\
&+\P_g[S\notin \B(g)\cup\cC(g)] \E_g \big[\psi_S(\alpha)^2 \mid S\notin \B(g)\cup\cC(g)\big].
\end{align*}
For the last term, note that
\[\limsup_{g\to\infty}\P_g[S\notin \B(g)\cup\cC(g)] \E_g \big[\psi_S(\alpha)^2 \mid S\notin \B(g)\cup\cC(g)\big]
\leq \lim_{g\to\infty}\P_g[S\notin \B(g)\cup\cC(g)] 4 =0.\]
Therefore,
\[
\lim_{g\to\infty} \E_g \big[\psi_S(\alpha)^2\big]
= (\gamma+\phi\alpha )^2 \left(\lim_{g\to\infty} \P_g[S\in \B(g)] 
+\P_g[S\in \cC(g)] \right)
=(\gamma+\phi\alpha )^2.\qedhere
\]
\end{proof}

We can combine these expectation results to show that $\psi_S(\alpha)$ is concentrating around~$\psi(\alpha)$.

\begin{corollary}\label{Cor: small alpha exp concentration}
Suppose that $0<\alpha<\frac{1}{\sqrt{5}}$.
Then we have
\[\lim_{g\to\infty} \E_g \big[ \big(\psi_S(\alpha) -(\gamma+\phi\alpha) \big)^2\big] = 0. \]
\end{corollary}
\begin{proof}
Note that
\begin{align*}
\lim_{g\to\infty} \E_g \big[ \big(\psi_S(\alpha) -(\gamma+\phi\alpha) \big)^2\big]
&=\lim_{g\to\infty} \E_g \big[ \psi_S(\alpha)^2\big]
-2(\gamma+\phi\alpha)\lim_{g\to\infty} \E_g \big[ \psi_S(\alpha) \big]
+ (\gamma+\phi\alpha)^2\\
&=(\gamma+\phi\alpha)^2
-2(\gamma+\phi\alpha)(\gamma+\phi\alpha)
+ (\gamma+\phi\alpha)^2 =0. \qedhere
\end{align*}
\end{proof}

Finally, we can complete the case of Theorem~\ref{Thm: pointwise convergence} where $0<\alpha<\frac{1}{\sqrt{5}}$ using Chebyshev's inequality.

\begin{corollary}\label{Cor: pointwise small alpha}
Given $0<\alpha<\frac{1}{\sqrt{5}}$ and $\epsilon>0$, we have
\[\lim_{g\to\infty} \P_g[ |\psi_S(\alpha) -\psi(\alpha)|<\epsilon ] =1.\]
\end{corollary}
\begin{proof}
By Chebyshev's inequality, we have
\[
\P_g\big(|\psi_S(\alpha)-\psi(\alpha)|\geq \epsilon\big)
\leq \frac{\E_g[(\psi_S(\alpha)-\psi(\alpha))^2]}{\epsilon^2}.
\]
Applying Corollary~\ref{Cor: small alpha exp concentration} completes the proof.
\end{proof}

\begin{proof}[Proof of Theorem~\ref{Thm: pointwise convergence}]
This follows from Proposition~\ref{Prop: pointwise alpha 0}, Proposition~\ref{Prop: pointwise large alpha}, and Corollary~\ref{Cor: pointwise small alpha}.
\end{proof}

\section{Numerical semigroups ordered by Frobenius number: The proof of Theorem \ref{Thm: Uniform Convergence by frob}}\label{Sec: proof for frob}

The argument for numerical semigroups ordered by Frobenius number follows the same structure as the argument for numerical semigroups ordered by genus. Theorem \ref{Thm: Uniform Convergence by frob} follows from a pointwise convergence result.

\begin{theorem}\label{Thm: Frob pointwise convergence}
Given $\alpha\in[0,1]\setminus\{\frac{1}{4}\}$ and $\epsilon>0$, we have
\[
\lim_{F\to\infty} \P_{\Frob,F}[ |\Xi_S(\alpha) -\Xi(\alpha)|<\epsilon ] =1.
\]
\end{theorem}

The majority of this section is devoted to proving this theorem.  Before we go into the proof, we explain how this theorem leads to the proof of Theorem \ref{Thm: Uniform Convergence by frob}.

\begin{proof}[Proof of Theorem \ref{Thm: Uniform Convergence by frob} assuming Theorem \ref{Thm: Frob pointwise convergence}]

Consider $\epsilon>0$.
Choose $N$ such that $\frac{2}{2N+1} <\frac{\epsilon}{2}$. For each $i$ satisfying $0 \le i \le 2N+1$, let $\alpha_i=\frac{i}{2N+1}$. Note that $\Xi(\alpha_{i+1})-\Xi(\alpha_i)<\frac{\epsilon}{2}$.
Now given any $\alpha\in [0,1]$, there is some $i$ for which $\alpha_i\leq \alpha\leq \alpha_{i+1}$. Note that
\begin{eqnarray*}
\Xi_S(\alpha)-\Xi(\alpha) & \leq \Xi_S(\alpha_{i+1})-\Xi(\alpha_i) & < \Xi_S(\alpha_{i+1})-\Xi(\alpha_{i+1}) +\frac{\epsilon}{2},\\
\Xi(\alpha)-\Xi_S(\alpha) & \leq \Xi(\alpha_{i+1})-\Xi_S(\alpha_i)  &< \Xi(\alpha_{i})-\Xi_S(\alpha_{i}) +\frac{\epsilon}{2}.
\end{eqnarray*}
From this we can conclude that
\[\sup_{0\leq \alpha\leq 1} |\Xi_S(\alpha)-\Xi(\alpha)| \leq \max_{0\leq i\leq 2N+1} |\Xi_S(\alpha_i)- \Xi(\alpha_i)| +\frac{\epsilon}{2}.\]
This means that if there is some $\alpha\in[0,1]$ for which $|\Xi_S(\alpha)-\Xi(\alpha)|>\epsilon$, then there is some $0\leq i\leq 2N+1$ for which $|\Xi_S(\alpha_i)-\Xi(\alpha_i)|>\frac{\epsilon}{2}$. This implies that
\[
\P_{\Frob,F}
\Big[ \sup_{0\leq \alpha\leq 1} |\Xi_S(\alpha) -\Xi(\alpha)| >\epsilon \Big] 
\leq \sum_{i=0}^{2N+1} 
\P_{\Frob,F}
\Big[ |\Xi_S(\alpha_i) -\Xi(\alpha_i)| >\frac{\epsilon}{2} \Big].
\]

For each $0\leq i\leq 2N+1$, Theorem~\ref{Thm: Frob pointwise convergence} says that
\[
\lim_{F\to\infty} \P_{\Frob,F}
\Big[ |\Xi_S(\alpha_i) -\Xi(\alpha_i)| 
>\frac{\epsilon}{2} \Big] 
= 0.
\]
The result follows.
\end{proof}

The goal of the rest of this section is to prove Theorem \ref{Thm: Frob pointwise convergence}. We start with the cases where $\alpha=0$ and where $\frac{1}{4}<\alpha\leq 1$ because these are easier to handle.

\begin{proposition}\label{Prop: Frob pointwise alpha 0}
Let $\epsilon>0$.  We have
\[
\lim_{F\to\infty} \P_{\Frob,F}[ |\Xi_S(0) -\Xi(0)|<\epsilon ] =1.
\]
\end{proposition}
\begin{proof}
For any $S \in \SS_{\Frob,F}$ we have $\Xi_S(0) = \frac{a_{1}(S)}{F} = \frac{m(S)}{F}$. Therefore, Theorem~\ref{thm: backelin m=F/2} implies that
\[
\lim_{F\to\infty} \P_{\Frob,F}\left[ \left|\Xi_S(0) -\Xi(0)\right|<\epsilon \right]
=\lim_{F\to\infty} \P_{\Frob,F}\left[ \left|m(S) -\frac{1}{2} F\right|<\epsilon F \right]=1. \qedhere
\]
\end{proof}

\begin{proposition}\label{Prop: Frob pointwise large alpha}
Suppose $\alpha$ satisfies $\frac{1}{4}<\alpha\leq 1$.  Let $\epsilon>0$.  We have
\[
\lim_{F\to\infty} \P_{\Frob,F}[ |\Xi_S(\alpha) -\Xi(\alpha)|<\epsilon ] =1.
\]
\end{proposition}
\begin{proof}
Let $t=1+\lfloor\alpha (F-1)\rfloor$, so $\Xi_S(\alpha)= \frac{a_{t}(S)}{F}$. First, note that if $g(S)\geq F+1-t$, then Lemma~\ref{Lem: t>F+1-g} implies that $a_t(S) = t + g(S) = 1+\lfloor \alpha(g-1)\rfloor + g(S)$.  In this case,
\begin{align*}
|\Xi_S(\alpha) -\Xi(\alpha)|
&=\left|\frac{a_t(S)}{F} - \left(\frac{3}{4}+\alpha\right)\right|
= \left|\frac{1+\lfloor\alpha (F-1)\rfloor +g(S)}{F} - \left(\frac{3}{4}+\alpha\right)\right|\\
&\leq \left|\frac{1+\alpha (F-1) +g(S)}{F} - \left(\frac{3}{4}+\alpha\right)\right| +\frac{1}{F} 
<\frac{2}{F} +\left|\frac{g(S)}{F} - \frac{3}{4}\right|.
\end{align*}
Therefore, for $F>\frac{3}{\epsilon}$, we see that
\begin{align*}
\P_{\Frob,F}[ |\Xi_S(\alpha) -\Xi(\alpha)|<\epsilon ]
&\geq \P_{\Frob,F}\left[g(S)\geq F+1-t \text{ and } \left|\frac{g(S)}{F} - \frac{3}{4}\right| <\frac{\epsilon}{3} \right]\\
&\geq \P_{\Frob,F}\left[g(S)\geq (1-\alpha)F+2 \text{ and } \left|\frac{g(S)}{F} - \frac{3}{4}\right| <\frac{\epsilon}{3} \right].
\end{align*}

Since $1-\alpha<\frac{3}{4}$, the result follows from Theorem~\ref{thm: singhal genus by frob}.
\end{proof}

\subsection{Depth 2, $0<\alpha<\frac{1}{4}$}\label{Subsec: depth 2 by Frob}

By Proposition~\ref{Prop: depth 2 by frob}, if $S\in \B_{\Frob}(F)$ with $g(S)=g$, then $S=S_{F,B}$ for some subset $B\subseteq\{1,\dots,\lfloor\frac{F-1}{2}\rfloor\}$ of size $F-g$.
Note that $a_{k}(S)$ corresponds to the $k^{th}$ smallest element of $B$. Therefore, Theorem~\ref{Thm: order statistic} directly implies the following.

\begin{corollary}\label{Cor: ak depth 2 by Frob}
For $1\leq k\leq F-g$, we have
\begin{eqnarray*}
\E_{\Frob,F}[a_{k}(S)\mid S\in \B_{\Frob}(F,g)] & = & \left\lfloor\frac{F}{2}\right\rfloor+ \frac{\left(\left\lfloor\frac{F-1}{2}\right\rfloor+1\right)k}{F-g+1},\\
\Var_{\Frob, F}[a_{k}(S)\mid S\in \B_{\Frob}(F,g)] & = & \frac{\left(\left\lfloor\frac{F-1}{2}\right\rfloor+1\right)\left(\left\lfloor\frac{F-1}{2}\right\rfloor-F+g\right) k (F-g+1-k)}{(F+1-g)^2 (F+2-g)}.
\end{eqnarray*}
\end{corollary}

\begin{lemma}\label{Lem: exp, var depth 2 by Frob}
Assume $0<\epsilon_1<\tfrac18$ and $F>\tfrac{1}{\epsilon_1}$. Suppose $g$ and $\alpha$ satisfy $\big|g-\tfrac{3}{4} F\big|<\epsilon_1 F$ and $0< \alpha \le \tfrac14-2\epsilon_1$. Then
\[
\Bigl|\E_{\Frob,F}\big[\Xi_S(\alpha)\mid S\in\mathcal B_{\Frob}(F,g)\big]-\Big(\tfrac{1}{2}+2\alpha\Big)\Bigr|
\le 17\,\epsilon_1,
\]
and
\[
\Bigl|\E_{\Frob,F}\big[\Xi_S(\alpha)^2\mid S\in\mathcal B_{\Frob}(F,g)\big]-\Big(\tfrac{1}{2}+2\alpha\Big)^2\Bigr|
\le 563\,\epsilon_1.
\]
\end{lemma}

\begin{proof}
Let $t=1+\lfloor \alpha(F-1)\rfloor$ and set
\[
\kappa=\frac{t}{F},\qquad \delta=\frac{F-g}{F},\qquad h=\frac{\lfloor\frac{F}{2}\rfloor}{F},\qquad H=\frac{\lfloor\frac{F-1}{2}\rfloor+1}{F}.
\]
Note that $|\delta-\frac{1}{4}|< \epsilon_1$ by the hypothesis on $g$, and that $|\kappa-\alpha|\leq\frac{1}{F} <\epsilon_1 $.

By Corollary~\ref{Cor: ak depth 2 by Frob},
\[
\E_{\Frob,F}[a_t(S)\mid S\in \mathcal B_{\Frob}(F,g)]
=\Big\lfloor\frac{F}{2}\Big\rfloor+\frac{(\lfloor\frac{F-1}{2}\rfloor+1)t}{(F-g)+1}.
\]
Divide by $F$ to obtain
\[
\E_{\Frob,F}[\Xi_S(\alpha)\mid S\in \mathcal B_{\Frob}(F,g)]
=h+H\cdot\frac{\kappa}{\delta+\tfrac{1}{F}}.
\]
Hence
\[
\Bigl|\E_{\Frob,F}[\Xi_S(\alpha)\mid S\in \mathcal B_{\Frob}(F),\ g(S)=g]-\Big(\tfrac12+2\alpha\Big)\Bigr|
\le \Big|h-\tfrac12\Big|
+\Big|H\cdot\frac{\kappa}{\delta+\tfrac{1}{F}}-2\alpha\Big|.
\]
It is easy to see that $|h-\frac{1}{2}| \le \epsilon_1/2$. For the second term,
\[
\Big|H\cdot\frac{\kappa}{\delta+\tfrac{1}{F}}-2\alpha\Big|
\leq \Big|H-\tfrac12\Big|\cdot \frac{\kappa}{\delta+\tfrac{1}{F}}
+\frac12\Big|\frac{\kappa}{\delta+\tfrac{1}{F}}-4\alpha\Big|.
\]
Since $|H-\tfrac12|< \tfrac{\epsilon_1}{2}$, $\kappa\leq 1$, and $\delta +\frac{1}{F}\ge \frac{1}{4}-\epsilon_1\geq \frac{1}{8}$, for the first summand we have
\[
\Big|H-\tfrac12\Big|\cdot \frac{\kappa}{\delta+\tfrac{1}{F}}
\leq \frac{\epsilon_1}{2} \cdot \frac{1}{\left(\frac{1}{8}\right)}= 4\epsilon_1.
\]
For the remaining summand,
\[
\Big|\frac{\kappa}{\delta+\tfrac{1}{F}}-4\alpha\Big|
\le \frac{|\kappa-\alpha|}{\delta+\tfrac{1}{F}}
+\alpha\Big|\frac{1-4(\delta+\frac{1}{F})}{\delta+\tfrac{1}{F}}\Big|.
\]
Recall that $|\kappa-\alpha|< \epsilon_1$, $\delta+\tfrac{1}{F}\geq \frac{1}{8}$, $\alpha\leq \frac{1}{4}$, and $|1-4\delta|< 4\epsilon_1$.  We see that
\[
\frac{|\kappa-\alpha|}{\delta+\tfrac{1}{F}}
+\alpha\Big|\frac{1-4(\delta+\frac{1}{F})}{\delta+\tfrac{1}{F}}\Big|
\leq \frac{\epsilon_1}{\frac{1}{8}} + \frac{1}{4} \frac{4\epsilon_1 +4\epsilon_1}{\frac{1}{8}}
=24\epsilon_1.
\]
Putting the bounds together,
\[
\Bigl|\E_{\Frob,F}[\Xi_S(\alpha)\mid S\in \B_{\Frob}(F,g)]-\Big(\tfrac12+2\alpha\Big)\Bigr|
\le \frac{\epsilon_1}{2}+4\epsilon_1+12\epsilon_1
< 17 \epsilon_1.
\]

For the second statement, we write
\begin{align*}
&\E_{\Frob,F}[\Xi_S(\alpha)^2\mid S\in \B_{\Frob}(F,g)]\\
&=\Var_{\Frob,F}(\Xi_S(\alpha)\mid S\in \B_{\Frob}(F,g))+\E_{\Frob,F}[\Xi_S(\alpha)\mid S\in \B_{\Frob}(F,g)]^2.
\end{align*}
From Corollary~\ref{Cor: ak depth 2 by Frob},
\begin{align*}
\Var_{\Frob,F}(\Xi_S(\alpha)\mid S\in \B_{\Frob}(F,g))
&=\frac{\Var_{\Frob,F}(a_t(S)\mid S\in \B_{\Frob}(F,g))}{F^2}\\
&\le \frac{(\lfloor\frac{F-1}{2}\rfloor+1)(\lfloor\frac{F-1}{2}\rfloor-F+g) t (F-g+1-t)}{(F+1-g)^2 (F+2-g) F^2}.
\end{align*}
Since $g\leq F$ and $t\leq F$, we see that
\begin{align*}
\Var_{\Frob,F}(\Xi_S(\alpha)\mid S\in \B_{\Frob}(F,g))
&\leq \frac{F^4}{(F+1-g)^3 F^2} =\frac{1}{F} \frac{1}{(\delta+\frac{1}{F})^3} <8^3 \epsilon_1.
\end{align*}
Since $\Xi_S(\alpha)\in[0,2]$,
\begin{align*}
& \Bigl| \E_{\Frob,F}\big[\Xi_S(\alpha)^2 \mid S\in\mathcal B_{\Frob}(F,g)\big]
-\Big(\tfrac{1}{2}+2\alpha\Big)^2\Bigr|\\
= &\ \Bigl| \Var_{\Frob,F}(\Xi_S(\alpha)\mid S\in \B_{\Frob}(F,g))
+\E_{\Frob,F}[\Xi_S(\alpha)\mid S\in \B_{\Frob}(F,g)]^2
-\Big(\tfrac{1}{2}+2\alpha\Big)^2\Bigr|\\
\leq &\ \left| \Var_{\Frob,F}(\Xi_S(\alpha)\mid S\in \B_{\Frob}(F,g)) \right| \\
& +\bigl| \E_{\Frob,F}[\Xi_S(\alpha)\mid S\in \B_{\Frob}(F,g)]
-\big(\tfrac{1}{2}+2\alpha\big) \bigr| 
\cdot \bigl| \E_{\Frob,F}[\Xi_S(\alpha)\mid S\in \B_{\Frob}(F,g)]
+\big(\tfrac{1}{2}+2\alpha\big)\bigr|\\
< &\  512 \epsilon_1 + (17\epsilon_1) (2+1) =563\epsilon_1. \qedhere 
\end{align*}
\end{proof}

\begin{proposition}\label{Prop: small alpha depth 2 by frob}
Suppose that $0<\alpha<\frac{1}{4}$.
Then we have
\[
\lim_{F\to\infty} \E_{\Frob,F} \big[\Xi_S(\alpha)\mid S\in\mathcal B_{\Frob}(F)\big] = \frac{1}{2}+2\alpha. 
\]
Moreover,
\[\lim_{F\to\infty} \E_{\Frob,F} \big[\Xi_S(\alpha)^2 \mid S\in\mathcal B_{\Frob}(F)\big] = \Big(\frac{1}{2}+2\alpha\Big)^2. \]
\end{proposition}

\begin{proof}
Choose $\epsilon_1$ satisfying $0<\epsilon_1<\min\big(\tfrac18,\ \tfrac{1}{2}\big(\tfrac14-\alpha\big)\big)$.

Using the law of total expectation and then dividing up $\B_{\Frob}(F)$ based on the value of $g$, we see that
\begin{align*}
&\ \Bigl|\E_{\Frob,F}\big[\Xi_S(\alpha) \mid S\in\mathcal B_{\Frob}(F) \big]
-\Big(\tfrac12+2\alpha\Big)\Bigr|\\
=&\ \Bigl|
\sum_{g}\E_{\Frob,F}\big[\Xi_S(\alpha)\mid S\in\B_{\Frob}(F,g)\big]
   \P_{\Frob,F}\big(g(S)=g \mid S\in\mathcal B_{\Frob}(F)\big) 
   -\Big(\tfrac12+2\alpha\Big)\Bigr|\\
\le & \ \sum_{g: |g-\frac{3F}{4}|<\epsilon_1 F}
   \Bigl|\E_{\Frob,F}\big[\Xi_S(\alpha)\mid \B_{\Frob}(F,g)\big]
   -\Big(\tfrac12+2\alpha\Big)\Bigr|\;
   \P_{\Frob,F}\big(g(S)=g \mid S\in\mathcal B_{\Frob}(F)\big)\\
& +\sum_{g: |g-\frac{3F}{4}|\ge \epsilon_1 F}
   \Bigl|\E_{\Frob,F}\big[\Xi_S(\alpha)\mid \B_{\Frob}(F,g)\big]-\Big(\tfrac12+2\alpha\Big)\Bigr|\;
   \P_{\Frob,F}\big(g(S)=g \mid S\in\mathcal B_{\Frob}(F)\big).
\end{align*}
Lemma~\ref{Lem: exp, var depth 2 by Frob} and the fact that $\Xi_S(\alpha)\in [0,2]$, imply that
\begin{align*}
&\Bigl|\E_{\Frob,F}\big[\Xi_S(\alpha) \mid S\in\mathcal B_{\Frob}(F) \big]-\Big(\tfrac12+2\alpha\Big)\Bigr|\\
&\le \sum_{g: |g-\frac{3F}{4}|<\epsilon_1 F}
   17 \epsilon_1
   \P_{\Frob,F}\big(g(S)=g \mid S\in\mathcal B_{\Frob}(F)\big)
   +\sum_{g: |g-\frac{3F}{4}| \ge \epsilon_1 F}
   2   \P_{\Frob,F}\big(g(S)=g\mid S\in\mathcal B_{\Frob}(F)\big)\\
&= 17 \epsilon_1
   \P_{\Frob,F}\left(\left|g-\frac{3F}{4}\right|<\epsilon_1 F \mid S\in\mathcal B_{\Frob}(F)\right)
   +2   \P_{\Frob,F}\left(\left|g-\frac{3F}{4}\right| \ge \epsilon_1 F \mid S\in\mathcal B_{\Frob}(F)\right).
\end{align*}
By Theorem~\ref{thm: singhal genus by frob}, $\P_{\Frob,F}\big[|g-\frac{3F}{4}|\ge \epsilon_1 F \mid S\in\mathcal B_{\Frob}(F)\big]\to 0$ as $F\to\infty$.
Taking $\limsup_{F\to\infty}$ gives
\[
\limsup_{F\to\infty}\Bigl|\E_{\Frob,F}\big[\Xi_S(\alpha) \mid S\in\mathcal B_{\Frob}(F) \big]-\Big(\tfrac12+2\alpha\Big)\Bigr|
\le 17\epsilon_1.
\]
Since $\epsilon_1$ can be chosen to be arbitrarily small, we conclude
\[
\lim_{F\to\infty}\E_{\Frob,F}\big[\Xi_S(\alpha) \mid S\in\mathcal B_{\Frob}(F) \big]=\tfrac12+2\alpha.
\]

Next, for the second moment, exactly the same decomposition yields
\begin{align*}
&\ \Bigl|\E_{\Frob,F}\big[\Xi_S(\alpha)^2 \mid S\in\mathcal B_{\Frob}(F) \big]-\Big(\tfrac12+2\alpha\Big)^2\Bigr|\\
\le & \ \sum_{g: |g-\frac{3F}{4}|<\epsilon_1 F}
\Bigl|\E_{\Frob,F}\big[\Xi_S(\alpha)^2\mid \B_{\Frob}(F,g)\big]-\Big(\tfrac12+2\alpha\Big)^2\Bigr|
   \P_{\Frob,F}(g(S)=g \mid S\in\mathcal B_{\Frob}(F)\big)\\
& +\sum_{g: |g-\frac{3F}{4}| \ge \epsilon_1 F}
\Bigl|\E_{\Frob,F}\big[\Xi_S(\alpha)^2\mid \B_{\Frob}(F,g)\big]-\Big(\tfrac12+2\alpha\Big)^2\Bigr|
   \P_{\Frob,F}(g(S)=g\mid S\in\mathcal B_{\Frob}(F)\big).
\end{align*}
By Lemma~\ref{Lem: exp, var depth 2 by Frob} and the fact that $\Xi_S(\alpha)\in[0,2]$, we see that
\begin{align*}
&\ \Bigl|\E_{\Frob,F}\big[\Xi_S(\alpha)^2 \mid S\in\mathcal B_{\Frob}(F) \big]-\Big(\tfrac12+2\alpha\Big)^2\Bigr|\\
\le &\ 563\epsilon_1 \P_{\Frob,F}\left(\left|g-\frac{3F}{4}\right|<\epsilon_1 F \mid S\in\mathcal B_{\Frob}(F)\right)
+4\P_{\Frob,F}\left(\left|g-\frac{3F}{4}\right|>\epsilon_1 F \mid S\in\mathcal B_{\Frob}(F)\right).
\end{align*}
Taking $\limsup_{F\to\infty}$, using Theorem~\ref{thm: singhal genus by frob} and then choosing decreasing values of $\epsilon_1$ approaching $0$ gives
\[
\lim_{F\to\infty}\E_{\Frob,F}\big[\Xi_S(\alpha)^2 \mid S\in\mathcal B_{\Frob}(F) \big]=\Big(\tfrac12+2\alpha\Big)^2.
\]
This completes the proof.
\end{proof}

\subsection{Depth 3, $0<\alpha<\frac{1}{4}$}

In this section, we follow the same strategy as in Section~\ref{Subsec: depth 2 by Frob}, but focus on the semigroups in $\cC_{\Frob}(F)$. We start by applying Theorem~\ref{Thm: order statistic} in the setting of Proposition~\ref{Prop: parameterize C(F,m,g) by frob}.  

\begin{lemma}\label{Lem: aks depth 3 F}
Consider $(A_1,A_2)\in \Ima(\chi_{g,m,F-2m})$. For $|A_1|<k\leq F-g-|A_2|$, we have
\begin{align*}
&\E_{\Frob,F}[a_{k}(S)\mid S\in \cC_{\Frob}(F,g,m), \chi_{g,m,F-2m}(S)=(A_1,A_2)] \\
&= (F-m)+ \frac{(3m-F) (k-|A_1|)}{(F-g-|A_1|-|A_2|+1)},  
\end{align*}
and 
\begin{align*}
&\Var_{\Frob,F}[a_{k}(S)\mid S\in \cC_{\Frob}(F,g,m), \chi_{g,m,F-2m}(S)=(A_1,A_2)] \\
&= \frac{(3m-F)(3m+g-2F+|A_1|+|A_2|-1)(k-|A_1|)(F-g-|A_2|+1-k)}{(F-g-|A_1|-|A_2|+1)^2 (F-g-|A_1|-|A_2|+2)}.  
\end{align*}
\end{lemma}
\begin{proof}
By Proposition~\ref{Prop: parameterize C(F,m,g) by frob}, we know that  numerical semigroups $S$ with $\chi_{g,m,F-2m}(S) = (A_1, A_2)$ are in bijection with $B\subseteq \{1,2,\ldots,3m-F-1\}$ of size $F-g-|A_1|-|A_2|$. Also note that $a_k(S)$ corresponds to the $(k-|A_1|)^{th}$ element of $B$. The result follows from Theorem~\ref{Thm: order statistic}.
\end{proof}

\begin{lemma}\label{Lem: depth 3 fix A1 A2 by Frob}
Assume $\epsilon_1< \frac{1}{40}$ and $F>\frac{1}{\epsilon_1}$.
Suppose we are given $F,g$, and $m$ that satisfy $|m-\frac{1}{2}F|<\epsilon_1 F$ and $|g-\frac{3}{4}F|<\epsilon_1 F$.
Also suppose $(A_1,A_2)\in \Ima(\chi_{g,m,F-2m})$ and $\alpha$ satisfies $3\epsilon_1<\alpha\le \frac{1}{4}-4\epsilon_1$.
Then
\[
\left|\E_{\Frob,F} \left[\Xi_S(\alpha)\mid S\in \cC_{\Frob}(F,g,m), \chi_{g,m,F-2m}(S)=(A_1,A_2) \right]-\left(\frac{1}{2}+2\alpha\right)\right|
< 37\epsilon_1. 
\]
Moreover,
\[
\left|
\E_{\Frob,F} 
[\Xi_S(\alpha)^2\mid S\in \cC_{\Frob}(F,g,m), \chi_{g,m,F-2m}(S)=(A_1,A_2) ] - \left(\frac{1}{2}+2\alpha\right)^2 \right|
< 165\epsilon_1. 
\]
\end{lemma}
\begin{proof}
Let $t=1+\lfloor \alpha(F-1)\rfloor$ so that $\Xi_S(\alpha)=a_t(S)/F$. 
Define the scaled quantities
\[
r=\frac{m}{F},\qquad \delta=\frac{F-g}{F},\qquad 
a=\frac{|A_1|}{F},\qquad b=\frac{|A_2|}{F},\qquad 
\kappa=\frac{t}{F},\qquad D=\delta-a-b+\frac{1}{F}.
\]
From the hypotheses, we have 
\[
|r-\tfrac12|< \epsilon_1,\quad |\delta-\tfrac14|< \epsilon_1,\quad 
0\le a,b<  2\epsilon_1,\quad
|\kappa-\alpha|\le \frac{1}{F}<\epsilon_1.
\]
Using the fact that $\alpha$ satisfies $3\epsilon_1<\alpha\le \frac{1}{4}-4\epsilon_1$, together with the bounds on $a$ and $b$, we see that
$|A_1|<t\le F-g-|A_2|$.  Therefore, we can apply Lemma~\ref{Lem: aks depth 3 F}. 
Moreover, since $\epsilon_1<\tfrac1{40}$ we see that
\[
D=\delta-a-b+\tfrac1F \ge \big(\tfrac14-\epsilon_1\big)- 4\epsilon_1  \ge \tfrac18.
\]

Divide the expectation in Lemma~\ref{Lem: aks depth 3 F} by $F$ to get
\[
\E_{\Frob,F}\left[\Xi_S(\alpha)\mid S\in\cC_{\Frob}(F,g,m),  \chi_{g,m,F-2m}(S)=(A_1,A_2)\right]
=(1-r)+\frac{(3r-1)(\kappa-a)}{D}.
\]
Therefore,
\begin{align*}
&\ \Big| \E_{\Frob,F}\left[\Xi_S(\alpha)\mid S\in\cC_{\Frob}(F,g,m),  \chi_{g,m,F-2m}(S)=(A_1,A_2)\right] - (\tfrac12+2\alpha) \Big|\\
\leq & \ |(1-r)-\tfrac12| 
+ \Big|\Big(\tfrac{3r-1}{D}-2\Big)(\kappa-a)\Big|
+ 2\,|\kappa-\alpha| + 2a. 
\end{align*}

We bound the four terms in this sum.
\begin{itemize}
\item Recall that $|(1-r)-\tfrac12|=|r-\tfrac12|< \epsilon_1$.
\item We have $|\kappa-\alpha|< \epsilon_1$ and $a< 2\epsilon_1$.  Therefore, $2|\kappa-\alpha|+2a< 6\epsilon_1$.
\item Note that
\[
|\kappa-a|\leq \kappa+ a<(\alpha+\epsilon_1) +2\epsilon_1 <\tfrac{1}{4}.
\]
Next,
\begin{align*}
\Big|\tfrac{3r-1}{D}-2\Big|
&=\frac{\big|(3r-1)-2(\delta-a-b+\tfrac1F)\big|}{D}
\leq \frac{|3r-\frac{3}{2}| + |\frac{1}{2} -2\delta|+2a+2b+\tfrac{2}{F}}{\frac{1}{8}}\\
&< 8 \big( 3\epsilon_1 + 2\epsilon_1 +4\epsilon_1+4\epsilon_1+2\epsilon_1\big) =120\epsilon_1.
\end{align*}
We conclude that $\big|\big(\tfrac{3r-1}{D}-2\big)(\kappa-a)\big| < 30\epsilon_1$.
\end{itemize}
Summing the three bounds yields
\[
\Bigl|\E_{\Frob,F}\left[\Xi_S(\alpha)\mid \cC_{\Frob}(F,g,m),  \chi_{g,m,F-2m}\right]-\Big(\tfrac12+2\alpha\Big)\Bigr|
\le \epsilon_1+6\epsilon_1+30\epsilon_1 = 37\epsilon_1.
\]

Next, to compute the second moment, we divide the variance in Lemma~\ref{Lem: aks depth 3 F} by $F^2$ to get
\begin{align*}
&\ \Var_{\Frob,F}[\Xi_S(\alpha) \mid  \cC_{\Frob}(F,g,m), \chi_{g,m,F-2m}] \\
= &\ \frac{1}{F} \frac{(3r-1)(3r -\delta -1 +a+b-\frac{1}{F})(\kappa-a)(\delta-b+\frac{1}{F}-\kappa)}{D^2 (D+\frac{1}{F})}\\
< &\  \epsilon_1 \frac{(\tfrac{1}{2}+3\epsilon_1)((\frac{3}{2}+3\epsilon_1)   -1 +2\epsilon_1+2\epsilon_1)(\alpha+\epsilon_1)(\frac{1}{4}+\epsilon_1+\epsilon_1)}{(\tfrac{1}{8})^3}\\
< &\ 17\epsilon_1.
\end{align*}
In the last step, we used the fact that $\alpha<\frac{1}{4}$ and $\epsilon_1<\frac{1}{40}$.

Therefore,
\begin{align*}
&\ \Bigl|\E_{\Frob,F}\big[\Xi_S(\alpha)^2\mid \cC_{\Frob}(F,g,m), \chi_{g,m,F-2m}\big]-\Big(\tfrac12+2\alpha\Big)^2\Bigr|\\
\leq &\  \Var_{\Frob,F}\big(\Xi_S(\alpha)\mid \cC_{\Frob}(F,g,m), \chi_{g,m,F-2m}\big)\\
& + \Bigl|\E_{\Frob,F}\big[\Xi_S(\alpha)\mid \cC_{\Frob}(F,g,m), \chi_{g,m,F-2m}\big]^2 -\Big(\tfrac12+2\alpha\Big)^2\Bigr|\\
< &\ 17 \epsilon_1
+ \Bigl|\E_{\Frob,F}\big[\Xi_S(\alpha)\mid \cC_{\Frob}(F,g,m), \chi_{g,m,F-2m}\big]-\Big(\tfrac12+2\alpha\Big)\Bigr| \\
&\ \cdot\Bigl|\E_{\Frob,F}\big[\Xi_S(\alpha)\mid\cC_{\Frob}(F,g,m), \chi_{g,m,F-2m}\big]+\Big(\tfrac12+2\alpha\Big)\Bigr|.
\end{align*}
Since $\Xi_S(\alpha)\in[0,2]$, the absolute value of the second term in the final summand is at most $4$. Using the bound on the expectation that we proved earlier, we conclude that
\[
\Bigl|\E_{\Frob,F}\big[\Xi_S(\alpha)^2\mid \cC_{\Frob}(F,g,m), \chi_{g,m,F-2m}\big]-\Big(\tfrac12+2\alpha\Big)^2\Bigr|
< 17 \epsilon_1 + 4\cdot 37\epsilon_1 = 165\epsilon_1.
\]
This completes the proof.
\end{proof}

\begin{corollary}\label{Cor: depth 3 fix gml by Frob}
Assume $\epsilon_1< \frac{1}{40}$ and $F>\frac{1}{\epsilon_1}$.
Suppose we are given $F,g$, and $m$ that satisfy $|m-\frac{1}{2}F|<\epsilon_1 F$ and $|g-\frac{3}{4}F|<\epsilon_1 F$ and that $\alpha$ satisfies $3\epsilon_1<\alpha\le \frac{1}{4}-4\epsilon_1$.
Then
\[
\left|\E_{\Frob,F} \big[\Xi_S(\alpha)\mid S\in \cC_{\Frob}(F,g,m) \big]-\left(\tfrac{1}{2}+2\alpha\right)\right|
< 37\epsilon_1. 
\]
Moreover,
\[
\left|\E_{\Frob,F} \big[\Xi_S(\alpha)^2\mid S\in \cC_{\Frob}(F,g,m) \big]-\left(\tfrac{1}{2}+2\alpha\right)^2\right|
< 165\epsilon_1. 
\]    
\end{corollary}
\begin{proof}
By the law of total expectation, we can average over the $(A_1,A_2)$ in the image of $\chi_{g,m,F-2m}$. This shows 
\[
\E_{\Frob,F}\big[\Xi_S(\alpha)\mid S\in \cC_{\Frob}(F,g,m)\big]
=\E_{\Frob,F}\big[\E_{\Frob,F}\big[\Xi_S(\alpha)\mid S\in \cC_{\Frob}(F,g,m),\chi_{g,m,F-2m}\big]\big].
\]

From Lemma~\ref{Lem: depth 3 fix A1 A2 by Frob}, we have
\[
\Bigl|\E_{\Frob,F} \big[\Xi_S(\alpha)\mid S\in \cC_{\Frob}(F,g,m),\chi_{g,m,F-2m}(S)=(A_1,A_2)\big]
-(\tfrac{1}{2}+2\alpha)\Bigr|< 37\epsilon_1
\]
for every $(A_1,A_2)\in \operatorname{Im}(\chi_{g,m,F-2m})$. Averaging over $\chi_{g,m,F-2m}$ preserves
this bound.  This proves the first claim. The proof of the second claim is identical.
\end{proof}

We again use the law of total expectation, this time averaging over different $m,g$ in order to obtain the expectations conditional on $S\in\cC_{\Frob}(F)$.

\begin{proposition}\label{Prop: small alpha depth 3 by Frob}
Suppose that $0<\alpha<\frac{1}{4}$.
Then we have
\[
\lim_{F\to\infty} \E_{\Frob,F} \big[\Xi_S(\alpha)\mid S\in \cC_{\Frob}(F)\big] = \tfrac{1}{2}+2\alpha. 
\]
Moreover,
\[
\lim_{F\to\infty} \E_{\Frob,F} \big[\Xi_S(\alpha)^2\mid S\in \cC_{\Frob}(F)\big] = (\tfrac{1}{2}+2\alpha)^2. 
\]
\end{proposition}
\begin{proof}
Choose $\epsilon_1$ satisfying
\[
0<\epsilon_1<\min\left(\tfrac{1}{40},\tfrac{\alpha}{3},\tfrac{1}{4}(\tfrac{1}{4}-\alpha)\right),
\]
and take $F>1/\epsilon_1$. Consider the event
\[G_{\epsilon_1}=\big\{|m(S)-\tfrac{1}{2} F|<\epsilon_1 F\big\}\cap\big\{|g(S)-\tfrac{3}{4} F|<\epsilon_1 F\big\}.\]
By Theorem~\ref{thm: backelin m=F/2} and Theorem~\ref{thm: singhal genus by frob}, we know that
$\lim_{F\to\infty} \P_{\Frob,F}[G_{\epsilon_1}\mid S\in\cC_{\Frob}(F)]=1$.

We divide up $\cC_{\Frob}(F)$ by the pair $(m,g)$ and see that 
\begin{align*}
&\ \E_{\Frob,F}\big[\Xi_S(\alpha)\mid S\in\cC_{\Frob}(F)\big]\\
= &\ \sum_{m,g}\E_{\Frob,F}\big[\Xi_S(\alpha)\mid S\in\cC_{\Frob}(F,g,m)\big]
\P\big(S\in\cC_{\Frob}(F,g,m)\mid S\in\cC_{\Frob}(F)\big).
\end{align*}
For $(m,g)$ corresponding to $G_{\epsilon_1}$, Corollary~\ref{Cor: depth 3 fix gml by Frob} gives
\[
\Bigl|\E_{\Frob,F}\big[\Xi_S(\alpha)\mid S\in\cC_{\Frob}(F,g,m)\big]-(\tfrac{1}{2}+2\alpha)\Bigr|
< 37\epsilon_1.
\]
For $(m,g)$ corresponding to the complement $G_{\epsilon_1}^c$, we use the trivial bound
\[
\big|\E_{\Frob,F}[\Xi_S(\alpha)\mid S\in\cC_{\Frob}(F,g,m)]-(\tfrac{1}{2}+2\alpha)\big|\leq 2,
\]
which holds because $\Xi_S(\alpha)\in[0,2]$.
Therefore
\begin{align*}
&\ \Big|\E_{\Frob,F}[\Xi_S(\alpha)\mid S\in\cC_{\Frob}(F)]-(\tfrac{1}{2}+2\alpha)\Big|\\
<&\  37\epsilon_1\ \P(G_{\epsilon_1}\mid S\in\cC_{\Frob}(F))
+2\P(G_{\epsilon_1}^c\mid S\in\cC_{\Frob}(F)).
\end{align*}
Taking $\limsup_{F\to\infty}$ yields
\[
\limsup_{F\to\infty}\Big|\E_{\Frob,F}[\Xi_S(\alpha)\mid S\in\cC_{\Frob}(F)]-(\tfrac{1}{2}+2\alpha)\Big|
\leq 37\epsilon_1.
\]
Since $\epsilon_1$ can be arbitrarily small, we conclude that
\[\lim_{F\to\infty}\E_{\Frob,F}[\Xi_S(\alpha)\mid S\in\cC_{\Frob}(F)]=\tfrac{1}{2}+2\alpha.\]

A similar argument shows that
\begin{align*}
&\ \Big|\E_{\Frob,F}[\Xi_S(\alpha)^2\mid S\in\cC_{\Frob}(F)]-(\tfrac{1}{2}+2\alpha)^2\Big|\\
\leq &\ 165\epsilon_1\ \P(G_{\epsilon_1}\mid S\in\cC_{\Frob}(F))
+4\P(G_{\epsilon_1}^c\mid S\in \cC_{\Frob}(F)).
\end{align*}
Taking $\limsup_{F\to\infty}$ yields
\[
\limsup_{F\to\infty}\Big|\E_{\Frob,F}[\Xi_S(\alpha)^2\mid S\in\cC_{\Frob}(F)]-(\tfrac{1}{2}+2\alpha)^2\Big|
\leq 165\epsilon_1.\]
Since $\epsilon_1$ can be arbitrarily small, we conclude that
\[
\lim_{F\to\infty}\E_{\Frob,F}[\Xi_S(\alpha)^2\mid S\in\cC_{\Frob}(F)]=(\tfrac{1}{2}+2\alpha)^2.\qedhere
\]
\end{proof}

\subsection{Concluding $0<\alpha<\frac{1}{4}$}

We combine Proposition~\ref{Prop: small alpha depth 2 by frob} and Proposition~\ref{Prop: small alpha depth 3 by Frob} to obtain the unconditional expectations.

\begin{proposition}\label{Prop: small alpha by Frob}
Suppose that $0<\alpha<\frac{1}{4}$.
Then we have
\[\lim_{F\to\infty} \E_{\Frob,F} \big[\Xi_S(\alpha)\big] = \tfrac{1}{2}+2\alpha. \]
Moreover,
\[\lim_{F\to\infty} \E_{\Frob,F} \big[\Xi_S(\alpha)^2\big] = (\tfrac{1}{2}+2\alpha)^2. \]
\end{proposition}
\begin{proof}
Notice that
\begin{align*}
&\ \lim_{F\to\infty} \E_{\Frob,F} \big[\Xi_S(\alpha)\big]\\ 
= &\ \lim_{F\to\infty} \P_{\Frob,F}[S\in \B_{\Frob}(F)] \E_{\Frob,F} \big[\Xi_S(\alpha) \mid S\in \B_{\Frob}(F)\big]\\
& +\lim_{F\to\infty}\P_{\Frob,F}[S\in \cC_{\Frob}(F)] \E_{\Frob,F} \big[\Xi_S(\alpha) \mid S\in \cC_{\Frob}(F)\big]\\
& +\lim_{F\to\infty}\P_{\Frob,F}[S\notin \B_{\Frob}(F)\cup\cC_{\Frob}(F)] \E_{\Frob,F} \big[\Xi_S(\alpha) \mid S\notin \B_{\Frob}(F)\cup\cC_{\Frob}(F)\big].
\end{align*}
For the last term, note that
\begin{align*}
&\ \limsup_{F\to\infty}\P_{\Frob,F}[S\notin \B_{\Frob}(F)\cup\cC_{\Frob}(F)] \E_{\Frob,F} \big[\Xi_S(\alpha) \mid S\notin \B_{\Frob}(F)\cup\cC_{\Frob}(F)\big]\\
\leq &\ \lim_{F\to\infty}\P_{\Frob,F}[S\notin \B_{\Frob}(F)\cup\cC_{\Frob}(F)] \cdot 2.
\end{align*}
Theorem \ref{thm: backelin m=F/2} implies that 
\[
\lim_{F\to\infty}\P_{\Frob,F}[S\notin \B_{\Frob}(F)\cup\cC_{\Frob}(F)] = 0.
\]
Therefore, Proposition~\ref{Prop: small alpha depth 2 by frob} and Proposition~\ref{Prop: small alpha depth 3 by Frob} imply that 
\begin{align*}
&\ \lim_{F\to\infty} \E_{\Frob,F} \big[\Xi_S(\alpha)\big]\\
= &\  (\tfrac{1}{2}+2\alpha ) \left(\lim_{F\to\infty} \P_{\Frob,F}[S\in \B_{\Frob}(F)] 
+\P_{\Frob,F}[S\in \cC_{\Frob}(F)] \right)
=(\tfrac{1}{2}+2\alpha ).
\end{align*}

Similarly, we have
\begin{align*}
&\ \lim_{F\to\infty} \E_{\Frob,F} \big[\Xi_S(\alpha)^2\big]\\ 
= &\  \lim_{F\to\infty} \P_{\Frob,F}[S\in \B_{\Frob}(F)] \E_{\Frob,F} \big[\Xi_S(\alpha)^2 \mid S\in \B_{\Frob}(F)\big]\\
& +\lim_{F\to\infty}\P_{\Frob,F}[S\in \cC_{\Frob}(F)] \E_{\Frob,F} \big[\Xi_S(\alpha)^2 \mid S\in \cC_{\Frob}(F)\big]\\
& +\lim_{F\to\infty}\P_{\Frob,F}[S\notin \B_{\Frob}(F)\cup\cC_{\Frob}(F)] \E_{\Frob,F} \big[\Xi_S(\alpha)^2 \mid S\notin \B_{\Frob}(F)\cup\cC_{\Frob}(F)\big].
\end{align*}
For the last term, note that
\begin{align*}
&\ \limsup_{F\to\infty}\P_{\Frob,F}[S\notin \B_{\Frob}(F)\cup\cC_{\Frob}(F)] \E_{\Frob,F} \big[\Xi_S(\alpha)^2 \mid S\notin \B_{\Frob}(F)\cup\cC_{\Frob}(F)\big]\\
\leq &\  \lim_{F\to\infty}\P_{\Frob,F}[S\notin \B_{\Frob}(F)\cup\cC_{\Frob}(F)] \cdot 4 =0.
\end{align*}
Therefore, Proposition~\ref{Prop: small alpha depth 2 by frob} and Proposition~\ref{Prop: small alpha depth 3 by Frob} imply that 
\begin{align*}
&\ \lim_{F\to\infty} \E_{\Frob,F} \big[\Xi_S(\alpha)^2\big]\\
= &\  (\tfrac{1}{2}+2\alpha )^2 \left(\lim_{F\to\infty} \P_{\Frob,F}[S\in \B_{\Frob}(F)] 
+\P_{\Frob,F}[S\in \cC_{\Frob}(F)] \right)
=(\tfrac{1}{2}+2\alpha )^2.\qedhere
\end{align*}
\end{proof}

We combine these expectation results to show that $\Xi_S(\alpha)$ is concentrating around~$\Xi(\alpha)$.

\begin{corollary}\label{Cor: small alpha exp concentration by Frob}
Suppose that $0<\alpha<\frac{1}{4}$.
Then we have
\[\lim_{F\to\infty} \E_{\Frob,F} \big[ \big(\Xi_S(\alpha) -(\tfrac12+2\alpha) \big)^2\big] = 0. \]
\end{corollary}
\begin{proof}
Note that
\begin{align*}
&\ \lim_{F\to\infty} \E_{\Frob,F} \big[ \big(\Xi_S(\alpha) -(\tfrac12+2\alpha) \big)^2\big]\\
=&\ \lim_{F\to\infty} \E_{\Frob,F} \big[ \Xi_S(\alpha)^2\big]
-2(\tfrac12+2\alpha)\lim_{F\to\infty} \E_{\Frob,F} \big[ \Xi_S(\alpha) \big]
+ (\tfrac12+2\alpha)^2\\
=&\ (\tfrac12+2\alpha)^2
-2(\tfrac12+2\alpha)(\tfrac12+2\alpha)
+ (\tfrac12+2\alpha)^2 = 0. \qedhere
\end{align*}
\end{proof}

Finally, we can complete the case of Theorem~\ref{Thm: Frob pointwise convergence} where $0<\alpha<\frac{1}{4}$ using Chebyshev's inequality.

\begin{corollary}\label{Cor: pointwise small alpha by Frob}
Given $0<\alpha<\frac{1}{4}$ and $\epsilon>0$, we have
\[
\lim_{F\to\infty} \P_{\Frob,F}[ |\Xi_S(\alpha) -\Xi(\alpha)|<\epsilon ] =1.
\]
\end{corollary}
\begin{proof}
By Chebyshev's inequality, we have
\[
\P_{\Frob,F}\big(|\Xi_S(\alpha)-\Xi(\alpha)|\geq \epsilon\big)
\leq \frac{\E_{\Frob,F}[(\Xi_S(\alpha)-\Xi(\alpha))^2]}{\epsilon^2}.
\]
Applying Corollary~\ref{Cor: small alpha exp concentration by Frob} completes the proof.
\end{proof}

\begin{proof}[Proof of Theorem~\ref{Thm: Frob pointwise convergence}]
This follows from Proposition~\ref{Prop: Frob pointwise alpha 0}, Proposition~\ref{Prop: Frob pointwise large alpha}, and Corollary~\ref{Cor: pointwise small alpha by Frob}.
\end{proof}

\section*{Acknowledgments}
The first author was supported by project PID2024-156636NB-C21 (MATSE) funded by MCIN/AEI/10.13039/501100011033/ FEDER, UE.
The second author was supported by NSF Grant DMS 2154223.

\appendix
\section{Data related to regression parameters}\label{sec:appendix}

\vspace{-.3in}

\begin{table}
  \input{taula.tex}
  \caption{Regression parameters for numerical semigroups ordered by genus}
\label{tab:regressionparameters}
\end{table}
\begin{table}
  \input{taulaincrements.tex}
  \caption{Incremental regression parameters for numerical semigroups ordered by genus}
  \label{tab:regressionincrements}
\end{table}

\begin{table}
  \input{taula_xi.tex}
  \caption{Regression parameters for numerical semigroups ordered by Frobenius number}
\label{tab:regressionparameters_xi}
\end{table}
\begin{table}
  \input{taulaincrements_xi.tex}
  \caption{Incremental regression parameters for numerical semigroups ordered by Frobenius number}
  \label{tab:regressionincrements_xi}
\end{table}

\newpage

\bibliographystyle{plain}
\bibliography{Bibl}

\end{document}

%% file: taula.tex
\resizebox{\textwidth}{!}{\begin{tabular}{|r|rrrrrrrrr|}\hline$g$ & $n_g$ & $i_{cut}$ & $g-i_{cut}$ & $\bar m^l(g)$ & $\bar b^l(g)$ & $\bar {(R^2)}^l(g)$ & $\bar m^r(g)$ & $\bar b^r(g)$ & $\bar {(R^2)}^r(g)$ \\\hline 
4 &7 
& 2
& 2
&  1.285714 & 0.892857 &1.000000	
& 1.071429 & 0.928571 &1.000000	

\\
5 &12 
& 2
& 3
&  1.466667 & 0.833333 &1.000000	
& 1.000000 & 1.000000 &0.988095	

\\
6 &23 
& 3
& 3
&  1.467391 & 0.786232 &0.966939	
& 1.014493 & 0.983092 &0.987578	

\\
7 &39 
& 3
& 4
&  1.593407 & 0.761294 &0.965950	
& 1.028571 & 0.973993 &0.982316	

\\
8 &67 
& 4
& 4
&  1.564552 & 0.752425 &0.957808	
& 1.017351 & 0.987500 &0.985160	

\\
9 &118 
& 4
& 5
&  1.645951 & 0.732109 &0.955766	
& 1.027495 & 0.976836 &0.987199	

\\
10 &204 
& 5
& 5
&  1.601471 & 0.732255 &0.954763	
& 1.014706 & 0.988922 &0.989516	

\\
11 &343 
& 5
& 6
&  1.670289 & 0.722767 &0.952277	
& 1.028587 & 0.979024 &0.988436	

\\
12 &592 
& 5
& 7
&  1.731912 & 0.710755 &0.950128	
& 1.038660 & 0.969816 &0.989217	

\\
13 &1001 
& 6
& 7
&  1.676346 & 0.716836 &0.951423	
& 1.023164 & 0.983503 &0.991262	

\\
14 &1693 
& 6
& 8
&  1.728643 & 0.709253 &0.949598	
& 1.035185 & 0.973858 &0.990803	

\\
15 &2857 
& 7
& 8
&  1.679839 & 0.714435 &0.951261	
& 1.022708 & 0.984348 &0.992106	

\\
16 &4806 
& 7
& 9
&  1.719642 & 0.709242 &0.949995	
& 1.029914 & 0.978396 &0.992391	

\\
17 &8045 
& 8
& 9
&  1.675454 & 0.715232 &0.951975	
& 1.019532 & 0.987345 &0.993567	

\\
18 &13467 
& 8
& 10
&  1.711125 & 0.710539 &0.950840	
& 1.027168 & 0.980972 &0.993342	

\\
19 &22464 
& 9
& 10
&  1.671552 & 0.716219 &0.952838	
& 1.017476 & 0.989125 &0.994303	

\\
20 &37396 
& 9
& 11
&  1.700746 & 0.712718 &0.951991	
& 1.023471 & 0.984139 &0.994328	

\\
21 &62194 
& 9
& 12
&  1.729594 & 0.709395 &0.951187	
& 1.029903 & 0.978808 &0.994136	

\\
22 &103246 
& 10
& 12
&  1.691816 & 0.714753 &0.953233	
& 1.021040 & 0.986160 &0.994988	

\\
23 &170963 
& 10
& 13
&  1.716208 & 0.712118 &0.952650	
& 1.026231 & 0.981873 &0.994878	

\\
24 &282828 
& 11
& 13
&  1.682236 & 0.717158 &0.954597	
& 1.018202 & 0.988518 &0.995647	

\\
25 &467224 
& 11
& 14
&  1.704460 & 0.714708 &0.954113	
& 1.023333 & 0.984239 &0.995458	

\\
26 &770832 
& 12
& 14
&  1.674005 & 0.719391 &0.955996	
& 1.016200 & 0.990139 &0.996108	

\\
27 &1270267 
& 12
& 15
&  1.693339 & 0.717259 &0.955629	
& 1.020505 & 0.986535 &0.995981	

\\
28 &2091030 
& 13
& 15
&  1.665828 & 0.721578 &0.957419	
& 1.014018 & 0.991891 &0.996561	

\\
29 &3437839 
& 13
& 16
&  1.683657 & 0.719608 &0.957133	
& 1.018233 & 0.988361 &0.996392	

\\
30 &5646773 
& 13
& 17
&  1.700195 & 0.717881 &0.956877	
& 1.022071 & 0.985165 &0.996250	

\\
31 &9266788 
& 14
& 17
&  1.674522 & 0.721850 &0.958645	
& 1.015988 & 0.990149 &0.996769	

\\
32 &15195070 
& 14
& 18
&  1.689834 & 0.720223 &0.958450	
& 1.019744 & 0.987015 &0.996607	

\\
33 &24896206 
& 15
& 18
&  1.666496 & 0.723880 &0.960115	
& 1.014185 & 0.991569 &0.997075	

\\
34 &40761087 
& 15
& 19
&  1.680236 & 0.722403 &0.959990	
& 1.017482 & 0.988810 &0.996937	

\\
35 &66687201 
& 16
& 19
&  1.658973 & 0.725778 &0.961566	
& 1.012391 & 0.992978 &0.997358	

\\
36 &109032500 
& 16
& 20
&  1.671843 & 0.724367 &0.961493	
& 1.015623 & 0.990271 &0.997209	

\\
37 &178158289 
& 17
& 20
&  1.652351 & 0.727485 &0.962978	
& 1.010943 & 0.994100 &0.997590	

\\
38 &290939807 
& 17
& 21
&  1.664034 & 0.726186 &0.962951	
& 1.013799 & 0.991702 &0.997461	

\\
39 &474851445 
& 17
& 22
&  1.675331 & 0.724972 &0.962931	
& 1.016778 & 0.989218 &0.997316	

\\
40 &774614284 
& 18
& 22
&  1.657180 & 0.727812 &0.964366	
& 1.012292 & 0.992873 &0.997672	

\\
41 &1262992840 
& 18
& 23
&  1.667487 & 0.726686 &0.964380	
& 1.014962 & 0.990641 &0.997542	

\\
42 &2058356522 
& 19
& 23
&  1.650779 & 0.729317 &0.965730	
& 1.010806 & 0.994027 &0.997869	

\\
43 &3353191846 
& 19
& 24
&  1.660562 & 0.728218 &0.965779	
& 1.013425 & 0.991835 &0.997737	

\\
44 &5460401576 
& 20
& 24
&  1.645132 & 0.730657 &0.967049	
& 1.009581 & 0.994967 &0.998036	

\\
45 &8888486816 
& 20
& 25
&  1.654148 & 0.729625 &0.967122	
& 1.011943 & 0.992987 &0.997917	

\\
46 &14463633648 
& 21
& 25
&  1.639848 & 0.731892 &0.968315	
& 1.008363 & 0.995903 &0.998192	

\\
47 &23527845502 
& 21
& 26
&  1.648465 & 0.730878 &0.968413	
& 1.010683 & 0.993957 &0.998073	

\\
48 &38260496374 
& 22
& 26
&  1.635178 & 0.732987 &0.969536	
& 1.007356 & 0.996667 &0.998327	

\\
49 &62200036752 
& 22
& 27
&  1.643179 & 0.732027 &0.969649	
& 1.009455 & 0.994903 &0.998218	

\\
50 &101090300128 
& 22
& 28
&  1.650897 & 0.731121 &0.969764	
& 1.011657 & 0.993064 &0.998101	

\\
51 &164253200784 
& 23
& 28
&  1.638484 & 0.733047 &0.970837	
& 1.008413 & 0.995698 &0.998346	

\\
52 &266815155103 
& 23
& 29
&  1.645687 & 0.732183 &0.970964	
& 1.010431 & 0.994010 &0.998237	
\\\hline\end{tabular}}

%% file: taulaincrements.tex
\resizebox{\textwidth}{!}{\begin{tabular}{|r|rrrrrrrrr|}\hline$g$ & $n_g$ & $\Delta i_{cut}$ & $\Delta (g-i_{cut})$ & $\Delta\bar m^l(g)$ & $\Delta \bar b^l(g)$ & $\Delta \bar {(R^2)}^l(g)$ & $\Delta \bar m^r(g)$ & $\Delta \bar b^r(g)$ &$\Delta \bar {(R^2)}^r(g)$ \\\hline 
4 &7 
&2&-1
&1.285714	&0.892857	&1.000000	
&1.071429	&0.928571	&1.000000
\\
5 &12 
&0&1
&0.180952	&-0.059524	&0.000000	
&-0.071429	&0.071429	&-0.011905
\\
6 &23 
&1&0
&0.000725	&-0.047101	&-0.033061	
&0.014493	&-0.016908	&-0.000518
\\
7 &39 
&0&1
&0.126015	&-0.024938	&-0.000988	
&0.014079	&-0.009099	&-0.005261
\\
8 &67 
&1&0
&-0.028854	&-0.008869	&-0.008143	
&-0.011221	&0.013507	&0.002844
\\
9 &118 
&0&1
&0.081399	&-0.020316	&-0.002042	
&0.010145	&-0.010664	&0.002039
\\
10 &204 
&1&0
&-0.044480	&0.000146	&-0.001003	
&-0.012789	&0.012085	&0.002317
\\
11 &343 
&0&1
&0.068818	&-0.009488	&-0.002486	
&0.013881	&-0.009898	&-0.001079
\\
12 &592 
&0&1
&0.061623	&-0.012013	&-0.002149	
&0.010074	&-0.009208	&0.000781
\\
13 &1001 
&1&0
&-0.055566	&0.006082	&0.001294	
&-0.015497	&0.013687	&0.002045
\\
14 &1693 
&0&1
&0.052297	&-0.007583	&-0.001824	
&0.012021	&-0.009645	&-0.000459
\\
15 &2857 
&1&0
&-0.048804	&0.005182	&0.001663	
&-0.012477	&0.010490	&0.001303
\\
16 &4806 
&0&1
&0.039803	&-0.005193	&-0.001266	
&0.007205	&-0.005952	&0.000285
\\
17 &8045 
&1&0
&-0.044188	&0.005991	&0.001980	
&-0.010382	&0.008949	&0.001176
\\
18 &13467 
&0&1
&0.035672	&-0.004693	&-0.001135	
&0.007637	&-0.006373	&-0.000225
\\
19 &22464 
&1&0
&-0.039574	&0.005680	&0.001997	
&-0.009692	&0.008153	&0.000961
\\
20 &37396 
&0&1
&0.029194	&-0.003501	&-0.000846	
&0.005994	&-0.004986	&0.000025
\\
21 &62194 
&0&1
&0.028848	&-0.003323	&-0.000804	
&0.006433	&-0.005331	&-0.000193
\\
22 &103246 
&1&0
&-0.037778	&0.005358	&0.002046	
&-0.008864	&0.007352	&0.000852
\\
23 &170963 
&0&1
&0.024392	&-0.002635	&-0.000583	
&0.005192	&-0.004287	&-0.000110
\\
24 &282828 
&1&0
&-0.033972	&0.005040	&0.001947	
&-0.008030	&0.006644	&0.000769
\\
25 &467224 
&0&1
&0.022224	&-0.002450	&-0.000485	
&0.005131	&-0.004279	&-0.000189
\\
26 &770832 
&1&0
&-0.030455	&0.004684	&0.001883	
&-0.007133	&0.005900	&0.000649
\\
27 &1270267 
&0&1
&0.019334	&-0.002132	&-0.000367	
&0.004305	&-0.003603	&-0.000127
\\
28 &2091030 
&1&0
&-0.027511	&0.004319	&0.001790	
&-0.006488	&0.005355	&0.000580
\\
29 &3437839 
&0&1
&0.017829	&-0.001970	&-0.000286	
&0.004216	&-0.003530	&-0.000169
\\
30 &5646773 
&0&1
&0.016538	&-0.001727	&-0.000255	
&0.003838	&-0.003196	&-0.000141
\\
31 &9266788 
&1&0
&-0.025673	&0.003969	&0.001767	
&-0.006082	&0.004984	&0.000518
\\
32 &15195070 
&0&1
&0.015312	&-0.001627	&-0.000194	
&0.003756	&-0.003134	&-0.000161
\\
33 &24896206 
&1&0
&-0.023338	&0.003657	&0.001664	
&-0.005559	&0.004553	&0.000468
\\
34 &40761087 
&0&1
&0.013740	&-0.001477	&-0.000125	
&0.003296	&-0.002758	&-0.000138
\\
35 &66687201 
&1&0
&-0.021263	&0.003374	&0.001576	
&-0.005091	&0.004168	&0.000420
\\
36 &109032500 
&0&1
&0.012870	&-0.001411	&-0.000073	
&0.003232	&-0.002707	&-0.000148
\\
37 &178158289 
&1&0
&-0.019492	&0.003118	&0.001485	
&-0.004680	&0.003829	&0.000381
\\
38 &290939807 
&0&1
&0.011683	&-0.001299	&-0.000027	
&0.002856	&-0.002399	&-0.000129
\\
39 &474851445 
&0&1
&0.011296	&-0.001214	&-0.000020	
&0.002979	&-0.002484	&-0.000145
\\
40 &774614284 
&1&0
&-0.018151	&0.002840	&0.001435	
&-0.004486	&0.003655	&0.000356
\\
41 &1262992840 
&0&1
&0.010307	&-0.001126	&0.000014	
&0.002670	&-0.002232	&-0.000129
\\
42 &2058356522 
&1&0
&-0.016708	&0.002631	&0.001350	
&-0.004156	&0.003386	&0.000326
\\
43 &3353191846 
&0&1
&0.009783	&-0.001099	&0.000049	
&0.002620	&-0.002192	&-0.000132
\\
44 &5460401576 
&1&0
&-0.015430	&0.002439	&0.001271	
&-0.003844	&0.003132	&0.000299
\\
45 &8888486816 
&0&1
&0.009016	&-0.001033	&0.000072	
&0.002362	&-0.001980	&-0.000119
\\
46 &14463633648 
&1&0
&-0.014300	&0.002267	&0.001193	
&-0.003580	&0.002917	&0.000275
\\
47 &23527845502 
&0&1
&0.008617	&-0.001014	&0.000099	
&0.002320	&-0.001946	&-0.000119
\\
48 &38260496374 
&1&0
&-0.013287	&0.002109	&0.001122	
&-0.003326	&0.002709	&0.000253
\\
49 &62200036752 
&0&1
&0.008001	&-0.000960	&0.000113	
&0.002099	&-0.001764	&-0.000108
\\
50 &101090300128 
&0&1
&0.007718	&-0.000905	&0.000115	
&0.002202	&-0.001840	&-0.000118
\\
51 &164253200784 
&1&0
&-0.012412	&0.001925	&0.001072	
&-0.003244	&0.002635	&0.000245
\\
52 &266815155103 
&0&1
&0.007203	&-0.000863	&0.000127	
&0.002018	&-0.001689	&-0.000109
\\\hline\end{tabular}}

%% file: taula_xi.tex
\resizebox{\textwidth}{!}{\begin{tabular}{|r|rrrrrrrrr|}\hline$ F $ & $N_{\Frob}(F)$ & $i_{cut}$ & $ F-i_{cut}$ & $\bar m^l(F)$ & $\bar b^l(F)$ & $\bar {(R^2)}^l(F)$ & $\bar m^r(F)$ & $\bar b^r(F)$ & $\bar {(R^2)}^r(F)$ \\\hline 


 6 &4 
& 2
& 4
&  1.458333 & 0.833333 &1.000000	
& 0.958333 & 1.037500 &0.982857

\\
 7 &11 
& 2
& 5
&  1.714286 & 0.623377 &1.000000	
& 1.044156 & 0.877922 &0.988300

\\
 8 &10 
& 2
& 6
&  1.837500 & 0.675000 &1.000000	
& 1.007500 & 0.962738 &0.980808

\\
 9 &21 
& 3
& 6
&  1.925926 & 0.600529 &0.966858	
& 1.023129 & 0.898010 &0.991398

\\
 10 &22 
& 3
& 7
&  2.086364 & 0.609091 &0.956358	
& 0.994968 & 0.952273 &0.989251

\\
 11 &51 
& 3
& 8
&  2.067736 & 0.541295 &0.962284	
& 1.066760 & 0.838023 &0.987062

\\
 12 &40 
& 3
& 9
&  2.119792 & 0.617014 &0.953956	
& 1.016354 & 0.937257 &0.986615

\\
 13 &106 
& 4
& 9
&  2.056023 & 0.530552 &0.948962	
& 1.035559 & 0.851314 &0.990741

\\
 14 &103 
& 4
& 10
&  2.211442 & 0.547018 &0.938991	
& 1.032875 & 0.894145 &0.990005

\\
 15 &200 
& 4
& 11
&  2.090200 & 0.536133 &0.945807	
& 1.055048 & 0.842785 &0.988844

\\
 16 &205 
& 4
& 12
&  2.232165 & 0.543933 &0.937656	
& 1.048455 & 0.876095 &0.987145

\\
 17 &465 
& 5
& 12
&  2.097306 & 0.512030 &0.943624	
& 1.039802 & 0.836205 &0.991560

\\
 18 &405 
& 5
& 13
&  2.189712 & 0.542716 &0.934183	
& 1.031675 & 0.886218 &0.991197

\\
 19 &961 
& 5
& 14
&  2.139909 & 0.504628 &0.943059	
& 1.055825 & 0.818779 &0.990011

\\
 20 &900 
& 5
& 15
&  2.223211 & 0.526389 &0.933581	
& 1.051009 & 0.857127 &0.988997

\\
 21 &1828 
& 6
& 15
&  2.092559 & 0.514289 &0.944260	
& 1.035338 & 0.836675 &0.992786

\\
 22 &1913 
& 6
& 16
&  2.203631 & 0.516283 &0.934991	
& 1.040263 & 0.856807 &0.991969

\\
 23 &4096 
& 6
& 17
&  2.139864 & 0.499129 &0.942994	
& 1.050776 & 0.813499 &0.991339

\\
 24 &3578 
& 6
& 18
&  2.188234 & 0.525278 &0.935928	
& 1.045955 & 0.853953 &0.990347

\\
 25 &8273 
& 7
& 18
&  2.102635 & 0.502115 &0.945601	
& 1.035485 & 0.821431 &0.993618

\\
 26 &8175 
& 7
& 19
&  2.188947 & 0.509813 &0.937302	
& 1.039380 & 0.845005 &0.992933

\\
 27 &16132 
& 7
& 20
&  2.116653 & 0.503339 &0.945828	
& 1.043930 & 0.815455 &0.992485

\\
 28 &16267 
& 7
& 21
&  2.189563 & 0.510975 &0.937497	
& 1.046221 & 0.836609 &0.991602

\\
 29 &34903 
& 8
& 21
&  2.102207 & 0.498740 &0.947577	
& 1.033406 & 0.815456 &0.994346

\\
 30 &31822 
& 8
& 22
&  2.154607 & 0.514868 &0.941738	
& 1.033625 & 0.844727 &0.993825

\\
 31 &70854 
& 8
& 23
&  2.116989 & 0.497146 &0.947131	
& 1.040685 & 0.806638 &0.993442

\\
 32 &68681 
& 8
& 24
&  2.173456 & 0.507283 &0.940311	
& 1.042912 & 0.828964 &0.992642

\\
 33 &137391 
& 9
& 24
&  2.085983 & 0.502476 &0.951195	
& 1.029119 & 0.815762 &0.995036

\\
 34 &140661 
& 9
& 25
&  2.151000 & 0.506594 &0.943576	
& 1.033341 & 0.831844 &0.994454

\\
 35 &292081 
& 9
& 26
&  2.104389 & 0.497812 &0.949488	
& 1.036273 & 0.804231 &0.994253

\\
 36 &270258 
& 9
& 27
&  2.142265 & 0.510676 &0.945067	
& 1.037400 & 0.828431 &0.993529

\\
 37 &591443 
& 10
& 27
&  2.083303 & 0.498930 &0.952513	
& 1.027151 & 0.808211 &0.995590

\\
 38 &582453 
& 10
& 28
&  2.135169 & 0.505870 &0.946430	
& 1.030437 & 0.826306 &0.995066

\\
 39 &1156012 
& 10
& 29
&  2.087565 & 0.500594 &0.953180	
& 1.031925 & 0.804466 &0.994911

\\
 40 &1161319 
& 10
& 30
&  2.134425 & 0.506606 &0.946844	
& 1.034914 & 0.820295 &0.994284

\\
 41 &2425711 
& 11
& 30
&  2.075097 & 0.499340 &0.954754	
& 1.024522 & 0.805221 &0.996070

\\
 42 &2287203 
& 11
& 31
&  2.112609 & 0.508974 &0.950857	
& 1.026517 & 0.825342 &0.995611

\\
 43 &4889434 
& 11
& 32
&  2.083076 & 0.498733 &0.954543	
& 1.029057 & 0.799498 &0.995492

\\
 44 &4785671 
& 11
& 33
&  2.121519 & 0.505964 &0.949563	
& 1.031758 & 0.816138 &0.994914

\\
 45 &9575167 
& 12
& 33
&  2.061495 & 0.502038 &0.958259	
& 1.021339 & 0.804792 &0.996506

\\
 46 &9678844 
& 12
& 34
&  2.104942 & 0.506486 &0.952462	
& 1.024691 & 0.818535 &0.996056

\\
 47 &19919902 
& 12
& 35
&  2.073779 & 0.499488 &0.956837	
& 1.026020 & 0.797512 &0.995984

\\
 48 &18896892 
& 12
& 36
&  2.101580 & 0.508222 &0.953743	
& 1.028038 & 0.815173 &0.995455

\\
 49 &40010851 
& 13
& 36
&  2.057738 & 0.500886 &0.959433	
& 1.019613 & 0.800396 &0.996853

\\
 50 &39445886 
& 13
& 37
&  2.092948 & 0.506769 &0.955043	
& 1.022247 & 0.814991 &0.996459

\\
 51 &78794277 
& 13
& 38
&  2.060660 & 0.501712 &0.960165	
& 1.022891 & 0.797266 &0.996405

\\
 52 &78930306 
& 13
& 39
&  2.094451 & 0.506729 &0.955321	
& 1.025721 & 0.810273 &0.995932	
\\\hline\end{tabular}}

%% file: taulaincrements_xi.tex
\resizebox{\textwidth}{!}{\begin{tabular}{|r|rrrrrrrrr|}\hline$ F $ & $N_{\Frob}(F)$ & $\Delta i_{cut}$ & $\Delta (F-i_{cut})$ & $\Delta\bar m^l(F)$ & $\Delta \bar b^l(F)$ & $\Delta \bar {(R^2)}^l(F)$ & $\Delta \bar m^r(F)$ & $\Delta \bar b^r(F)$ &$\Delta \bar {(R^2)}^r(F)$ \\\hline 
 7 &11 
&0&1
&0.255952	&-0.209957	&0.000000
&0.085823	&-0.159578	&0.005443
\\
 8 &10 
&0&1
&0.123214	&0.051623	&0.000000
&-0.036656	&0.084816	&-0.007493
\\
 9 &21 
&1&0
&0.088426	&-0.074471	&-0.033142
&0.015629	&-0.064729	&0.010591
\\
 10 &22 
&0&1
&0.160438	&0.008562	&-0.010500
&-0.028162	&0.054263	&-0.002147
\\
 11 &51 
&0&1
&-0.018627	&-0.067796	&0.005926
&0.071793	&-0.114250	&-0.002189
\\
 12 &40 
&0&1
&0.052055	&0.075719	&-0.008328
&-0.050406	&0.099234	&-0.000447
\\
 13 &106 
&1&0
&-0.063768	&-0.086462	&-0.004994
&0.019205	&-0.085943	&0.004126
\\
 14 &103 
&0&1
&0.155419	&0.016467	&-0.009971
&-0.002684	&0.042831	&-0.000736
\\
 15 &200 
&0&1
&-0.121242	&-0.010885	&0.006816
&0.022173	&-0.051360	&-0.001161
\\
 16 &205 
&0&1
&0.141965	&0.007800	&-0.008151
&-0.006593	&0.033310	&-0.001699
\\
 17 &465 
&1&0
&-0.134859	&-0.031903	&0.005968
&-0.008653	&-0.039891	&0.004416
\\
 18 &405 
&0&1
&0.092406	&0.030686	&-0.009441
&-0.008127	&0.050013	&-0.000363
\\
 19 &961 
&0&1
&-0.049803	&-0.038088	&0.008876
&0.024150	&-0.067439	&-0.001186
\\
 20 &900 
&0&1
&0.083302	&0.021761	&-0.009478
&-0.004816	&0.038348	&-0.001014
\\
 21 &1828 
&1&0
&-0.130652	&-0.012100	&0.010679
&-0.015671	&-0.020452	&0.003789
\\
 22 &1913 
&0&1
&0.111072	&0.001994	&-0.009269
&0.004925	&0.020131	&-0.000817
\\
 23 &4096 
&0&1
&-0.063767	&-0.017154	&0.008002
&0.010513	&-0.043308	&-0.000630
\\
 24 &3578 
&0&1
&0.048370	&0.026149	&-0.007066
&-0.004821	&0.040454	&-0.000992
\\
 25 &8273 
&1&0
&-0.085599	&-0.023163	&0.009673
&-0.010470	&-0.032522	&0.003270
\\
 26 &8175 
&0&1
&0.086312	&0.007698	&-0.008299
&0.003895	&0.023574	&-0.000685
\\
 27 &16132 
&0&1
&-0.072294	&-0.006474	&0.008526
&0.004549	&-0.029550	&-0.000448
\\
 28 &16267 
&0&1
&0.072910	&0.007636	&-0.008331
&0.002291	&0.021154	&-0.000883
\\
 29 &34903 
&1&0
&-0.087356	&-0.012235	&0.010079
&-0.012815	&-0.021153	&0.002744
\\
 30 &31822 
&0&1
&0.052400	&0.016129	&-0.005838
&0.000220	&0.029270	&-0.000521
\\
 31 &70854 
&0&1
&-0.037618	&-0.017722	&0.005392
&0.007060	&-0.038088	&-0.000383
\\
 32 &68681 
&0&1
&0.056467	&0.010137	&-0.006819
&0.002227	&0.022326	&-0.000800
\\
 33 &137391 
&1&0
&-0.087474	&-0.004807	&0.010884
&-0.013793	&-0.013202	&0.002394
\\
 34 &140661 
&0&1
&0.065018	&0.004118	&-0.007619
&0.004222	&0.016081	&-0.000582
\\
 35 &292081 
&0&1
&-0.046611	&-0.008782	&0.005912
&0.002933	&-0.027613	&-0.000201
\\
 36 &270258 
&0&1
&0.037876	&0.012864	&-0.004421
&0.001127	&0.024200	&-0.000724
\\
 37 &591443 
&1&0
&-0.058962	&-0.011746	&0.007446
&-0.010249	&-0.020220	&0.002061
\\
 38 &582453 
&0&1
&0.051866	&0.006940	&-0.006083
&0.003287	&0.018095	&-0.000524
\\
 39 &1156012 
&0&1
&-0.047604	&-0.005276	&0.006750
&0.001487	&-0.021841	&-0.000155
\\
 40 &1161319 
&0&1
&0.046860	&0.006012	&-0.006336
&0.002989	&0.015829	&-0.000627
\\
 41 &2425711 
&1&0
&-0.059329	&-0.007267	&0.007910
&-0.010392	&-0.015074	&0.001786
\\
 42 &2287203 
&0&1
&0.037512	&0.009635	&-0.003897
&0.001995	&0.020121	&-0.000459
\\
 43 &4889434 
&0&1
&-0.029532	&-0.010241	&0.003686
&0.002540	&-0.025844	&-0.000119
\\
 44 &4785671 
&0&1
&0.038442	&0.007230	&-0.004980
&0.002701	&0.016640	&-0.000577
\\
 45 &9575167 
&1&0
&-0.060024	&-0.003926	&0.008696
&-0.010419	&-0.011346	&0.001592
\\
 46 &9678844 
&0&1
&0.043447	&0.004448	&-0.005797
&0.003353	&0.013743	&-0.000450
\\
 47 &19919902 
&0&1
&-0.031163	&-0.006997	&0.004375
&0.001328	&-0.021023	&-0.000073
\\
 48 &18896892 
&0&1
&0.027802	&0.008733	&-0.003094
&0.002019	&0.017661	&-0.000529
\\
 49 &40010851 
&1&0
&-0.043842	&-0.007335	&0.005690
&-0.008426	&-0.014777	&0.001398
\\
 50 &39445886 
&0&1
&0.035210	&0.005883	&-0.004390
&0.002635	&0.014595	&-0.000394
\\
 51 &78794277 
&0&1
&-0.032288	&-0.005057	&0.005122
&0.000644	&-0.017725	&-0.000054
\\
 52 &78930306 
&0&1
&0.033791	&0.005018	&-0.004844
&0.002830	&0.013007	&-0.000473
\\\hline\end{tabular}}